\documentclass[reqno,12pt]{amsart}

\usepackage[all]{xypic}
\usepackage{enumerate}
\usepackage{hyperref}
\usepackage{a4}
\usepackage{amssymb}
\usepackage{tikz}

\DeclareMathOperator{\img}{img}
\DeclareMathOperator{\coker}{coker}

\DeclareMathOperator{\rank}{rank}
\DeclareMathOperator{\codim}{codim}
\DeclareMathOperator{\rk}{rk}
\DeclareMathOperator{\tr}{tr}
\DeclareMathOperator{\ad}{ad}
\DeclareMathOperator{\id}{id}
\DeclareMathOperator{\Emb}{Emb}
\DeclareMathOperator{\Ann}{Ann}
\DeclareMathOperator{\Ham}{Ham}
\DeclareMathOperator{\Symp}{Symp}
\DeclareMathOperator{\Gr}{Gr}
\DeclareMathOperator{\Diff}{Diff}
\DeclareMathOperator{\eend}{end}
\DeclareMathOperator{\Rot}{Rot}

\newcommand\locconst{\textnormal{loc.const.}}
\newcommand\emb{\textnormal{emb}}
\newcommand\aug{\textnormal{aug}}
\newcommand\reg{\textnormal{reg}}
\newcommand\deco{\textnormal{deco}}
\newcommand\ex{\textnormal{ex}}
\newcommand\KKS{\textnormal{KKS}}
\newcommand\lin{\textnormal{lin}}
\newcommand\iso{\textnormal{iso}}

\newcommand\ham{\mathfrak{ham}}

\theoremstyle{plain}
  \newtheorem{theorem}{Theorem}[section]
  \newtheorem{corollary}[theorem]{Corollary}
  \newtheorem{lemma}[theorem]{Lemma}
  \newtheorem{proposition}[theorem]{Proposition}
  
\theoremstyle{definition}
\theoremstyle{remark}
  \newtheorem{remark}[theorem]{Remark}
  \newtheorem{example}[theorem]{Example}

\begin{document}

\title[A dual pair for the group of volume preserving diffeomorphisms]{A dual pair for the group of\\ volume preserving diffeomorphisms}

\author{Stefan Haller}

\address{Stefan Haller,
	Department of Mathematics,
	University of Vienna,
	Oskar-Morgenstern-Platz 1,
	1090 Vienna,
	Austria.}

\email{stefan.haller@univie.ac.at}

\author{Cornelia Vizman}

\address{Cornelia Vizman,
	Department of Mathematics,
	West University of Timi\c soara, 
	Bd. V. Parvan 4,
	300223 Timi\c soara, 
	Romania.}

\email{cornelia.vizman@e-uvt.ro}

\begin{abstract}
	We use cotangent bundles of spaces of smooth embeddings to construct symplectic dual pairs involving the group of volume preserving diffeomorphisms.
	Via symplectic reduction we obtain descriptions of coadjoint orbits of this group in terms of nonlinear Grassmannians of augmented submanifolds.
	For codimension one embeddings these submanifolds are further constrained to the leaves of isodrastic foliations with finite codimensions.
\end{abstract}




\maketitle


\section{Introduction}

	Coadjoint orbits play an important role in representation theory.
	Infinite dimensional groups whose coadjoint orbits have been studied rigorously include	(central extensions of) loop groups \cite{PS}, 
	the diffeomorphism group (with Bott-Virasoro extension) \cite{BFP}, the symplectic and Hamiltonian group \cite{W90,HV04,L09,izosimov,HV20,HV23}, 
	the group of volume preserving diffeomorphisms \cite{I96,HV04}, and the contact group \cite{HV22}.
	This paper aims at describing new coadjoint orbits of the group of volume preserving diffeomorphisms.

	The Euler equation being a Hamiltonian system on the cotangent bundle of the group of volume preserving diffeomorphisms,
	such coadjoint orbits can serve as phase spaces for singular vorticities in an ideal fluid \cite{MW83}.
	These add to the already known coadjoint orbits of singular vorticity configurations of codimension one (vortex filaments) 
	and two (vortex sheets) \cite{goldin1,goldin2, khesin, HV04, GBV23, cv}. 
	Such orbits can be described by nonlinear Grassmannians (i.e.~spaces of submanifolds) or their decorated version, 
	in contrast to the orbits presented here, that are described by augmented nonlinear Grassmannians.
	The difference between the two types of nonlinear Grassmannians is that the additional structure to a submanifold $N$ 
	lives on $N$ in the decorated case, while in the augmented case it lives in the ambient space along $N$.

	In this paper  we will apply the general principle that symplectic reduction on one leg of a dual pair of moment maps 
	for two commuting Hamiltonian Lie group actions leads to coadjoint orbits of the other group. 
	The dual pair we use involves the left action of the group of volume preserving diffeomorphisms,
	thus the symplectic reduction for the right group, which is the diffeomorphism group of reparametrizations,
	yields our coadjoint orbits of the group of volume preserving diffeomorphisms modeled by augmented nonlinear Grassmannians.

	The same principle was used in \cite{GBV19,cv}, where symplectic reduction on the right leg of the ideal fluid
	dual pair due to Marsden and Weinstein \cite{MW83} led to coadjoint orbits of the Hamiltonian
	group consisting of weighted isotropic submanifolds of the symplectic manifold \cite{L09,W90}.
	It was also used in \cite{HV22}, where symplectic reduction on the right leg of the dual pair for the contact group 
	provided a conceptual identification of nonlinear Grassmannians of weighted submanifolds with certain coadjoint orbits of the group of contact diffeomorphisms. 

	Let us now sketch the main results of this paper in more details.

	The space of smooth embeddings $\Emb(S,M)$ from a closed manifold $S$ into a manifold $M$ is a Fr\'echet manifold in a natural way.
	The natural actions of the Lie groups $\Diff_c(M)$ and $\Diff(S)$ on the regular cotangent bundle $T^*_\reg\Emb(S,M)$ commute.
	Both actions are Hamiltonian with equivariant moment maps.
	Restricting to the open subset $T^*_{\reg,\times}\Emb(S,M)$, one obtains a dual pair known as the EPDiff dual pair \cite{HM05,GBV12}.

	In this paper we consider a volume form $\mu$ on $M$ and restrict the left action to the group of volume preserving diffeomorphisms, $\Diff_c(M,\mu)$.
	The qualitative behavior depends very much on the codimension of $S$ in $M$.
	If the codimension is at least two, then the action of $\Diff_c(M,\mu)$ on the geometric objects relevant for our purpose turns out to be 
	as transitive as the action of the full group $\Diff_c(M)$.
	On the other hand, the action of $\Diff_c(M,\mu)$ on the space of codimension one submanifolds $N$ in $M$ is no longer locally transitive,
	as the connected components of $M\setminus N$ have invariant volume. 
	We will now elaborate on this dichotomy.

	If $\dim M-\dim S\geq2$, then the $\Diff_c(M,\mu)$ orbits in $T^*_{\reg,\times}\Emb(S,M)$ are open in the $\Diff_c(M)$ orbits.
	Consequently, in this case the EPDiff dual pair remains a dual pair if the left action is restricted to $\Diff_c(M,\mu)$.
	This permits to recognize $T^*_{\reg,\times}\Gr_S(M)$, an open subset in the regular cotangent bundle of the nonlinear Grassmannian, 
	as a coadjoint orbit of $\Diff_c(M,\mu)$ via symplectic reduction at the level zero along the right leg.
	These statements remain true if we further restrict the left action to the subgroup of exact volume preserving diffeomorphisms, $\Diff_{c,\ex}(M,\mu)$.

	If $\dim M-\dim S=1$, then the $\Diff_{c}(M,\mu)$ orbits give rise to a foliation of finite codimension on $\Emb(S,M)$, which will be called the isovolume foliation.
	Using the regular cotangent bundles of isovolume leaves in $\Emb(S,M)$, we obtain a dual pair involving $\Diff_{c}(M,\mu)$.
	Right leg symplectic reduction at certain (non-zero) levels leads to a description of some coadjoint orbits of $\Diff_{c}(M,\mu)$ 
	as spaces of curves in $M$, augmented by classes of 1-form densities.
	The orbits of $\Diff_{c,\ex}(M,\mu)$ provide another foliation on $\Emb(S,M)$, which is analogous to Weinstein's isodrastic foliation for symplectic manifolds.
	Its codimension is also finite, but larger than the codimension of the isovolume foliation, in general.
	We obtain an analogous dual pair involving $\Diff_{c,\ex}(M,\mu)$ and similar descriptions of certain coadjoint orbits of $\Diff_{c,\ex}(M,\mu)$.

	The remaining part of this paper is organized as follows.
	In Section~\ref{S:EPDiff} we recall the EPDiff dual pair and identify some coadjoint orbits of $\Diff_c(M)$ via symplectic reduction.
	In Section~\ref{S:EPDiffvol} we introduce the EPDiffvol dual pair for embeddings of codimension at least two and identify some coadjoint orbits of $\Diff_c(M,\mu)$ and $\Diff_{c,\ex}(M,\mu)$.
	In Section~\ref{S:foliations} we discuss two foliations on codimension one embeddings which will be used in the construction of the dual pairs in Section~\ref{S:codim.one}.
	In Section~\ref{S:orbits.codim.one} we use the latter dual pairs to identify further coadjoint orbits of $\Diff_c(M,\mu)$ and $\Diff_{c,\ex}(M,\mu)$,
	and give an application to singular vortex configurations in the plane.

\section{The EPDiff dual pair}\label{S:EPDiff}

	We consider a smooth manifold $M$ and a closed manifold $S$.
	Holm and Marsden \cite{HM05} observed that the commuting actions of $\Diff_c(M)$ and $\Diff(S)$ on the cotangent bundle of the manifold of embeddings $\Emb(S,M)$
	constitute a symplectic dual pair \cite{W83} with mutually symplectic orthogonal \cite{LM87} orbits.
	A detailed proof of the symplectic orthogonality has been given in \cite{GBV12}.

	In this section we will discuss right leg reduction in the EPDiff dual pair, i.e., reduction for the $\Diff(S)$ action.
	Reducing at the level zero we will recognize certain open subsets in $T^*_\reg\Gr_S(M)$, the regular cotangent bundle of the nonlinear Grassmannian, as coadjoint orbits of $\Diff_c(M)$. 
	We will also reduce at specific nonzero levels, provided $S$ is 1-dimensional.
	This leads to a description of some coadjoint orbits of $\Diff_c(M)$ as manifolds of augmented curves in $M$.

\subsection{EPDiff dual pair}\label{SS:EPDiff}

	We begin by recalling the EPDiff dual pair in more details.
	The space of smooth embeddings $\Emb(S,M)$ is well known to be an open subset in the Fr\'echet manifold $C^\infty(S,M)$.
	We define the regular part of its cotangent bundle by
	\begin{equation}\label{treg}
		T^*_\reg\Emb(S,M):=C^\infty_{\lin,\emb}(|\Lambda|_S^*,T^*M)
	\end{equation}
	where the right hand side denotes the space of all vector bundle homomorphisms $\Phi\colon|\Lambda|^*_S\to T^*M$
	sitting over smooth embeddings $\varphi\colon S\to M$, with $|\Lambda|_S$ denoting the line bundle of densities over $S$.
	It will often be convenient to identify this with a space of pairs $\Phi=(\varphi,\alpha)$ as follows:
	\[
		T^*_\reg\Emb(S,M)
		=\bigl\{(\varphi,\alpha):\varphi\in\Emb(S,M),\alpha\in\Gamma(|\Lambda|_S\otimes\varphi^*T^*M)\bigr\}.
	\]
	This space carries a natural structure of a Fr\'echet manifold such that the canonical map $T^*_\reg\Emb(S,M)\to\Emb(S,M)$ is a smooth vector bundle.
	The natural actions of the groups $\Diff(M)$ and $\Diff(S)$ on the regular cotangent bundle preserve the tautological 1-form and the corresponding weak symplectic form.	
	
	The total space of the EPDiff dual pair is the open subset of (fiberwise linear) embeddings $\Phi\colon|\Lambda|_S^*\to T^*M$ in $T^*_\reg\Emb(S,M)$, and will be denoted by
	\begin{equation}\label{E:varE}
		\mathcal E:=T^*_{\reg,\times}\Emb(S,M):=\Emb_\lin(|\Lambda|_S^*,T^*M).
	\end{equation}
	The corresponding description in terms of pairs reads
	\begin{equation}\label{treg2}
		\mathcal E=\bigl\{(\varphi,\alpha):\varphi\in\Emb(S,M),\alpha\in\Gamma(|\Lambda|_S\otimes\varphi^*T^*M)\text{ nowhere zero}\bigr\}.
	\end{equation}
	Clearly, $\mathcal E$ is invariant under the actions of $\Diff(M)$ and $\Diff(S)$.
	The restriction of the tautological 1-form to $\mathcal E$ will be denoted by $\theta^{\mathcal E}$.
	Hence,
	\[
		\theta^{\mathcal E}_{(\varphi,\alpha)}(\bar\xi)=\int_S\alpha(\xi)
	\]
	if $\bar\xi\in T_{(\varphi,\alpha)}\mathcal E$ is a tangent vector over $\xi\in T_\varphi\Emb(S,M)$.
	Because $\mathcal E$ is open, the restriction of the symplectic form, $\omega^{\mathcal E}=d\theta^{\mathcal E}$, is a weak symplectic form on $\mathcal E$.
	By invariance, the tautological 1-form provides equivariant moment maps for the actions of $\Diff_c(M)$ and $\Diff(S)$:
	\begin{equation}\label{hm}
		\mathfrak X_c(M)^*\xleftarrow{J_L^{\mathcal E}}\mathcal E\xrightarrow{J_R^{\mathcal E}}\Omega^1(S;|\Lambda|_S)\subseteq\mathfrak X(S)^*.
	\end{equation}
	By definition, $\langle J^{\mathcal E}_L,X\rangle=\theta^{\mathcal E}(\zeta^{\mathcal E}_X)$ for $X\in\mathfrak X_c(M)$, 
	and $\langle J^{\mathcal E}_R,Y\rangle=\theta^{\mathcal E}(\zeta^{\mathcal E}_Y)$ for $Y\in\mathfrak X(S)$, 
	where $\zeta^{\mathcal E}_X$ and $\zeta^{\mathcal E}_Y$ denote the infinitesimal actions of $\Diff_c(M)$ and $\Diff(S)$ on $\mathcal E$, respectively.
	More explicitly, these moment maps are given by
	\[
		\langle J^{\mathcal E}_L(\varphi,\alpha),X\rangle=\int_S\alpha(X\circ\varphi)
		\qquad\text{and}\qquad
		J^{\mathcal E}_R(\varphi,\alpha)=\varphi^*\alpha,
	\]
	for $(\varphi,\alpha)\in\mathcal E$ and $X\in\mathfrak X_c(M)$.	

	\begin{theorem}[EPDiff dual pair {\cite{HM05,GBV12}}]\label{epdiff}
		The equivariant moment maps in \eqref{hm} constitute a symplectic dual pair with mutually symplectic orthogonal orbits.
	\end{theorem}

	Recall here that a symplectic dual pair \cite{W83} for commuting Hamiltonian actions of Lie groups $G$ and $H$ on a symplectic manifold $Q$ consists of two equivariant moment maps
	\[
		\mathfrak g^*\xleftarrow{J_L}Q\xrightarrow{J_R}\mathfrak h^*
	\]
	such that the distributions $\ker(TJ_L)$ and $\ker(TJ_R)$ are symplectic orthogonal complements of one another.
	In finite dimensions the latter orthogonality assumption holds if and only if the $G$-orbits and the $H$-orbits are mutually symplectic orthogonal, i.e., iff 
	\begin{equation}\label{E:mso}
		\mathfrak g_Q^\perp=\mathfrak h_Q\qquad\text{and}\qquad\mathfrak h_Q^\perp=\mathfrak g_Q
	\end{equation}
	where $\mathfrak g_Q$ and $\mathfrak h_Q$ denote the distributions on $Q$ provided by the tangent spaces to the orbits of $G$ and $H$, respectively.
	In infinite dimensions the latter condition can be strictly stronger \cite{GBV15}.
	As $\ker(TJ_L)=\mathfrak g_Q^\perp$ and $\ker(TJ_R)=\mathfrak h_Q^\perp$, the two conditions in \eqref{E:mso} assert (formally) that 
	$H$ acts infinitesimally transitive on the level sets of $J_L$ and $G$ acts infinitesimally transitive on the level sets of $J_R$, respectively.

\subsection{Augmented nonlinear Grassmannians}\label{SS:aug}

	The coadjoint orbits we are about to describe are subsets of the augmented nonlinear Grassmannian
	\begin{equation}\label{caug}
		\Gr_S^\aug(M):=\bigl\{(N,\gamma):N\in\Gr_S(M),\gamma\in\Gamma(|\Lambda|_N\otimes T^*M|_N)\bigr\}.
	\end{equation}
	There is a unique structure of a Fr\'echet manifold on $\Gr_S^\aug(M)$ such that the canonical projection 
	\begin{equation}\label{E:pb}
		T_\reg^*\Emb(S,M)\to\Gr_S^\aug(M),{\quad(\varphi,\alpha)\mapsto(\varphi(S),\varphi_*\alpha)}
	\end{equation}
	is a smooth (locally trivial) principal $\Diff(S)$ bundle.
	We denote the nonlinear Grassmannian of submanifolds decorated with 1-form densities by
	\[
		\Gr_S^\deco(M):=\bigl\{(N,\rho_N):N\in\Gr_S(M),\rho_N\in\Omega^1(N;|\Lambda|_N)\bigr\}.
	\]
	Using the canonical identification
	\begin{equation}\label{gemb}
		\Gr_S^\deco(M)=\Emb(S,M)\times_{\Diff(S)}\Omega^1(S;|\Lambda|_S),
	\end{equation}
	we equip the decorated Grassmannian with a smooth structure.
	Hence, $\Gr_S^\deco(M)$ is a smooth vector bundle over $\Gr_S(M)$ with typical fiber $\Omega^1(S;|\Lambda|_S)$.
	The canonical map
	\begin{equation}\label{gtog}
		\Gr_S^\aug(M)\to\Gr_S^\deco(M),\quad(N,\gamma)\mapsto(N,\iota_N^*\gamma)
	\end{equation}
	is easily seen to be a locally trivial smooth fiber bundle.
	The fiber over $(N,\rho_N)$ is an affine space over the vector space $\Gamma(|\Lambda|_N\otimes\Ann(TN))$,
	where $\Ann(TN)$ denotes the annihilator of $TN$, i.e., the kernel of the vector bundle homomorphism $T^*M|_N\to T^*N$.
	We summarize these manifolds  and projections in the diagram 
	\begin{equation}\label{E:qwerty}
		T^*_\reg\Emb(S,M)\xrightarrow{\Diff(S)}\Gr_S^\aug(M)\xrightarrow{{\rm affine}}\Gr_S^\deco(M)\xrightarrow{\Omega^1(S;|\Lambda|_S)}\Gr_S(M),
	\end{equation}
	indicating the structure group and the typical fibers over the arrows.

	Clearly, the set 
	\begin{multline}\label{calg}
		\mathcal G:=\Gr_{S,\times}^\aug(M)
		\\
		:=\bigl\{(N,\gamma):N\in\Gr_S(M),\gamma\in\Gamma(|\Lambda|_N\otimes T^*M|_N)\text{ nowhere zero}\bigr\}
	\end{multline}
	is open in $\Gr_S^\aug(M)$.
	As the set $\mathcal E$ in \eqref{E:varE} is the preimage of $\mathcal G$ under the bundle projection in \eqref{E:pb}, 
	the latter restricts to a principal $\Diff(S)$ bundle $\mathcal E\to\mathcal G$.
	The left moment map $J_L^{\mathcal E}$ in \eqref{hm} descends to a $\Diff(M)$ equivariant injective map
	\begin{equation}\label{E:barJL}
		\bar J_L^{\mathcal E}\colon\mathcal G\to\mathfrak X_c(M)^*,\quad\langle\bar J_L^{\mathcal E}(N,\gamma),X\rangle
		=\int_N\gamma(X|_N),
	\end{equation}
	where $(N,\gamma)\in\mathcal G$ and $X\in\mathfrak X_c(M)$.

	For the symplectic reduction we fix $\rho\in\Omega^1(S;|\Lambda|_S)\subseteq\mathfrak X(S)^*$.
	We define the non-linear Grassmannian decorated with 1-form densities of type $\rho$ as 
	\begin{equation}\label{gero2}
		\Gr_{S,\rho}^\deco(M):=\bigl\{(N,\rho_N)\in\Gr_S^\deco(M):(N,\rho_N)\cong(S,\rho)\bigr\}.
	\end{equation}
	We have a canonical identification
	\begin{equation}\label{gero}
		\Gr_{S,\rho}^\deco(M)=\Emb(S,M)\times_{\Diff(S)}\mathcal O_\rho,
	\end{equation}
	where $\mathcal O_\rho$ denotes the coadjoint $\Diff(S)$ orbit of $\rho$.
	The preimage of $\Gr_{S,\rho}^\deco(M)$ under the projection in \eqref{gtog} is
	\[
		\Gr_{S,\rho}^\aug(M):=\bigl\{(N,\gamma)\in\Gr_S^\aug(M):(N,\iota_N^*\gamma)\cong(S,\rho)\bigr\}.
	\]
	The reduced symplectic space
	\begin{equation}\label{E:grh}
		\mathcal G_\rho:=(J_R^{\mathcal E})^{-1}(\mathcal O_\rho)/\Diff(S)=\bigl\{(N,\gamma)\in\mathcal G:(N,\iota_N^*\gamma)\cong(S,\rho)\bigr\}
	\end{equation}
	is an open subset of $\Gr_{S,\rho}^\aug(M)$.
	Theorem~\ref{epdiff} implies, at least formally, that $\Diff_c(M)$ acts locally transitively on $\mathcal G_\rho$ and, consequently,
	the map in \eqref{E:barJL} identifies certain unions of connected components of $\mathcal G_\rho$ with coadjoint orbits of $\Diff_c(M)$.

	In order to obtain a smooth structure on $\mathcal G_\rho$ we will restrict to very specific $\rho$ for which $\mathcal O_\rho$ is a (splitting) smooth submanifold of $\Omega^1(S;|\Lambda|_S)$.
	In this case, every space in the sequence
	\begin{equation}\label{E:rhrh}
		(J^{\mathcal E}_R)^{-1}(\mathcal O_\rho)\xrightarrow{\Diff(S)}
		\mathcal G_\rho\stackrel{\rm open}\subseteq
		\Gr_{S,\rho}^\aug(M)\xrightarrow{{\rm affine}}
		\Gr_{S,\rho}^\deco(M)\xrightarrow{\mathcal O_\rho}\Gr_S(M)
	\end{equation}
	is a (splitting) smooth submanifold in the corresponding term in \eqref{E:qwerty}, and the arrows are locally trivial fiber bundles with structure group and typical fibers as indicated.
	This follows readily from the description of $\Gr_{S,\rho}^\deco(M)$ as an associated bundle provided in \eqref{gero}.

\subsection{Reduction at zero}\label{SS:EPDiff.zero}

	We now discuss symplectic reduction for the right leg of the EPDiff dual pair described in Theorem~\ref{epdiff}, i.e., reduction for the $\Diff(S)$ action.
	Reducing at zero we will recognize open subsets in the regular cotangent bundle of the nonlinear Grassmannian $\Gr_S(M)$ as coadjoint orbits of the diffeomorphism group $\Diff_c(M)$.
	Cotangent bundle reduction at zero is discussed in \cite[Theorem~2.2.2]{MMOPR07} in the finite dimensional setting.
	Here we provide an analogous description for our infinite dimensional situation. 
	
	Let $\mathcal E^\iso=(J_R^{\mathcal E})^{-1}(0)$ denote the zero level of the right moment map.
	It is not hard to see that $\mathcal E^\iso$ consists of all fiberwise linear embeddings $\Phi\colon|\Lambda|^*_S\to T^*M$ with $\Phi^*\theta^{T^*M}=0$. 
	By contraction with the Euler vector field, this condition is shown to be equivalent to $\Phi$ being an isotropic embedding, i.e., $\Phi^*d\theta^{T^*M}=0$, whence the notation $\mathcal E^\iso$.
	We denote the reduced symplectic space by 
	\begin{multline*}
		\mathcal G^\iso
		=\mathcal E^\iso/\Diff(S)
		\\
		=\bigl\{(N,\gamma):N\in\Gr_S(M),\gamma\in\Gamma(|\Lambda|_N\otimes\Ann(TN)),\gamma\text{\ nowhere zero}\bigr\}.
	\end{multline*}
	Since this coincides with $\mathcal G_\rho$ for $\rho=0$ the discussion in the preceding section, cf.~\eqref{E:rhrh}, 
	shows that $\mathcal G^\iso$ is a splitting smooth submanifold in $\mathcal G$, 
	that $\mathcal E^\iso$ is a splitting smooth submanifold of $\mathcal E$, and that the forgetful map $\mathcal E^\iso\to\mathcal G^\iso$ is a smooth principal $\Diff(S)$ bundle.

	Restricting the 1-form $\theta^{\mathcal E}$ to $\mathcal E^\iso$ it becomes vertical and descends to a 1-form $\theta^\iso$ on $\mathcal G^\iso$ 
	such that $\omega^\iso=d\theta^\iso$ is the reduced symplectic form on $\mathcal G^\iso$.
	Via the $\Diff(M)$ equivariant identification
	\begin{equation}\label{E:Giso}
		\mathcal G^\iso=T^*_{\reg,\times}\Gr_S(M),
	\end{equation}
	the 1-form $\theta^\iso$ corresponds to the canonical 1-form on the cotangent bundle of the nonlinear Grassmannian.
	In particular, this is a symplectomorphism intertwining the moment maps of the $\Diff_c(M)$ actions.

	The map $\bar J^{\mathcal E}_L$ in \eqref{E:barJL} restricts to a $\Diff(M)$ equivariant injective map
	\begin{equation}\label{E:barJLiso}
		\bar J^\iso_L\colon\mathcal G^\iso\to\mathfrak X_c(M)^*.
	\end{equation}
	According to Theorem~\ref{epdiff}, the $\Diff_c(M)$ actions on $\mathcal E^\iso$ and $\mathcal G^\iso$ are infinitesimally transitive.
	Proceeding as in the proof of \cite[Proposition~2]{HV04} we see that these actions are locally transitive.
	Hence, \eqref{E:barJLiso} identifies each connected component of $\mathcal G^\iso$ with a coadjoint orbit of $\Diff_c(M)_0$, the connected component of $\Diff_c(M)$ containing the identity.
	Moreover, $(\bar J^\iso_L)^*\omega^\KKS=\omega^\iso$ where $\omega^\KKS$ denotes the Kirillov--Kostant--Souriau symplectic form on the coadjoint orbit.

	Summarizing these observations we obtain the following result.

	\begin{theorem}\label{T:t}
		Each connected component of $T^*_{\reg,\times}\Gr_S(M)$, 
		endowed with its canonical symplectic form and the cotangent lifted $\Diff_c(M)_0$ action, 
		is symplectomorphic to a coadjoint orbit of $\Diff_c(M)_0$ via the equivariant moment map 
		\[
			J\colon T^*_{\reg,\times}\Gr_S(M)\to\mathfrak X_c(M)^*,\quad
			\langle J(N,\gamma),X\rangle=\int_N\gamma(X|_N),
		\]
		where $\gamma\in\Gamma(|\Lambda|_N\otimes\Ann(TN))$ nowhere zero and $X\in\mathfrak X_c(M)$.
		Moreover, $J$ identifies certain unions of connected components of $T^*_{\reg,\times}\Gr_S(M)$ with coadjoint orbits of the full group $\Diff_c(M)$.
	\end{theorem}

\subsection{Reduction at nonzero levels}\label{SS:EPDiff.nonzero}
	
	Assuming $S$ to be a 1-dimensional closed manifold, we will now discuss reduction for the $\Diff(S)$ action of the EPDiff dual pair in Theorem~\ref{epdiff} 
	at a level $\rho\in\Omega^1(S;|\Lambda|_S)\subseteq\mathfrak X(S)^*$ which is nowhere zero.

	If $\rho$ is nowhere zero, then there exists a unique volume form $\nu$ on $S$ such that $\rho=|\nu|\otimes\nu$.
	In particular, such a $\rho$ gives rise to a Riemannian metric $\nu\otimes\nu$ and an orientation on $S$.
	The circumference $\int_S|\nu|$ will be referred to as the length of $\rho$.
	Clearly, $\rho$ can be recovered from the metric and the orientation.
	
	Suppose, for a moment, that $S$ is connected, hence diffeomorphic to the circle $S^1$.
	As $\rho$ provides a Riemannian metric and an orientation on $S$,
	the isotropy group of $\rho$ is isomorphic to the rotation group $\Rot(S^1)\cong SO(2)\cong S^1$.
	Moreover, the $\Diff(S)$ coadjoint orbit $\mathcal O_\rho$ is a level set consisting of all 1-form densities which have the same length as $\rho$.
	In particular, $\mathcal O_\rho$ is a splitting smooth submanifold of $\Omega^1(S;|\Lambda|_S)$.

	Returning to the general case, we may assume $S=\bigsqcup_{i=1}^kS_i$ where each connected component $S_i$ is diffeomorphic to $S^1$.
	Note that we have a splitting short exact sequence of groups
	\[
		1\to\prod_{i=1}^k\Diff(S_i)\to\Diff(S)\to\mathfrak S_k\to1,
	\]
	where $\mathfrak S_k$ denotes the group of permutations of the set $\{1,\dotsc,k\}\cong\pi_0(S)$.
	In particular, $\Diff(S)\cong\Diff(S^1)^k\rtimes\mathfrak S_k$, the semidirect product with respect to the standard action of $\mathfrak S_k$ on $\Diff(S^1)^k$ by permutation of the factors.

	Let $\rho_i\in\Omega^1(S_i;|\Lambda|_{S_i})$ denote the restriction of $\rho$ to the connected component $S_i$, 
	let $l_i$ denote the length of $\rho_i$, and put $\ell_\rho:=(l_1,\dotsc,l_k)\in\mathbb R^k$.
	We obtain a splitting short exact sequence
	\[
		1\to\prod_{i=1}^k\Diff(S_i,\rho_i)\to\Diff(S,\rho)\to\mathfrak S_k(\ell_\rho)\to1
	\]
	where $\mathfrak S_k(\ell_\rho)$ denotes the subgroup of permutations preserving $\ell_\rho$.
	Hence, the isotropy group of $\rho$ is isomorphic to the semidirect product 
	\begin{equation}\label{E:semidirect}
		\Diff(S,\rho)\cong (S^1)^k\rtimes\mathfrak S_k(\ell_\rho)
	\end{equation}
	for the permutation action of $\mathfrak S_k(\ell_\rho)$ on the $k$-torus $(S^1)^k$.

	Via the canonical identification $\Omega^1(S;|\Lambda|_S)=\prod_{i=1}^k\Omega^1(S_i;|\Lambda|_{S_i})$ we have 
	\[
		\mathcal O_\rho=\bigcup_{\sigma\in\mathfrak S_k}\prod_{i=1}^k\mathcal O_{\rho_{\sigma(i)}}
	\]
	where $\mathcal O_{\rho_i}$ denotes the $\Diff(S_i)$ orbit of $\rho_i$.
	As each $\mathcal O_{\rho_i}$ is a splitting smooth submanifold of $\Omega^1(S_i;|\Lambda|_{S_i})$, we conclude that 
	$\mathcal O_\rho$ is a splitting smooth submanifold in $\Omega^1(S;|\Lambda|_S)$ of codimension $k$. 
	As explained near the end of Section~\ref{SS:aug}, this implies that $\mathcal G_\rho$ in~\eqref{E:grh} is a splitting smooth submanifold of $\mathcal G$.
	Hence, $\mathcal G_\rho$ is a space of augmented closed curves in $M$ which admits a natural Fr\'echet manifold structure.

	The map in \eqref{E:barJL} restricts to a $\Diff(M)$ equivariant injective map
	\begin{equation}\label{E:barJLrho}
		\bar J^\rho_L\colon\mathcal G_\rho\to\mathfrak X_c(M)^*.
	\end{equation}
	According to Theorem~\ref{epdiff}, the $\Diff_c(M)$ action on $\mathcal G_\rho$ is infinitesimally transitive.
	Proceeding as in the proof of \cite[Proposition~2]{HV04} we see that this action is locally transitive.
	Hence, \eqref{E:barJLrho} identifies each connected component of $\mathcal G_\rho$ with a coadjoint orbit of $\Diff_c(M)_0$.

	The reduced symplectic form $\omega_\rho$ on $\mathcal G_\rho$ admits a concrete description, as in \cite[Theorem~2.3.12]{MMOPR07},
	with the help of a principal connection on the principal $\Diff(S)$ bundle $\Emb(S,M)\to\Gr_S(M)$.
	The most convenient one is the connection induced by a Riemannian metric on $M$:
	\[
		A\in\Omega^1(\Emb(S,M);\mathfrak X(S)),\quad A(v_\varphi)=T\varphi^{-1}\circ v_\varphi^\top,
	\]
	where $v_\varphi\in T_\varphi\Emb(S,M)=\Gamma(\varphi^*TM)$ and $\varphi^{-1}\colon\varphi(S)\to S$.
	Here $v_\varphi=v_\varphi^\top+v_\varphi^\perp$ denotes the decomposition into tangential and orthogonal part along $\varphi(S)$ provided by the Riemannian metric.

	The orthogonal decomposition $TM|_N=TN\oplus T^\perp N$ with respect to the chosen Riemannian metric on $M$ yields the decomposition $T^*M|_N=T^*N\oplus\Ann(TN)$.
	Consequently, each element $(N,\gamma)\in\mathcal G_\rho$ can be uniquely written as the sum of $(N,\gamma^\top)\in\Gr_{S,\rho}^\deco(M)$ 
	and $(N,\gamma^\perp)\in T^*_\reg\Gr_S(M)$, whence the identification
	\begin{multline*}
		\mathcal G_\rho=T^*_\reg\Gr_S(M)\times_{\Gr_S(M)}\Gr_{S,\rho}^\deco(M)
		\\
		=\left\{(N,\gamma,\rho_N):\begin{array}{c}N\in\Gr_S(M),\ \gamma\in\Gamma(|\Lambda|_N\otimes\Ann(TN))\\ \rho_N\in\Omega^1(N;|\Lambda|_N),\ (N,\rho_N)\cong(S,\rho)\end{array}\right\}.
	\end{multline*}
	The reduced symplectic form becomes
	\begin{equation}\label{E:orr}
		\omega_\rho=-d\theta^{T^*_\reg\Gr_S(M)}-\beta,
	\end{equation}
	where the pullback of $\beta\in\Omega^2(\Gr_{S,\rho}^\deco(M))$ to $\Emb(S,M)\times\mathcal O_\rho$ via \eqref{gero} 
	is the sum of the coadjoint orbit symplectic form on $\mathcal O_\rho$ and $d\alpha$, where
	\[
		\alpha\in\Omega^1(\Emb(S,M)\times\mathcal O_\rho),\quad\alpha_{(\varphi,\nu)}(v_\varphi,\ad^*_\xi\nu)=\int_S(\nu,A(v_\varphi)).
	\]

	We summarize these observations in the following result.

	\begin{theorem}\label{T:EPDiff.red1dim}
		In the situation described above, each connected component of $T^*_\reg\Gr_S(M)\times_{\Gr_S(M)}\Gr_{S,\rho}^\deco(M)$, 
		endowed with the symplectic form in \eqref{E:orr} and the natural $\Diff_c(M)_0$ action,
		is symplectomorphic to a coadjoint orbit of $\Diff_c(M)_0$ via the equivariant moment map 
		\begin{gather*}
			J\colon T^*_\reg\Gr_S(M)\times_{\Gr_S(M)}\Gr_{S,\rho}^\deco(M)\to\mathfrak X_c(M)^*
			\\
			\langle J(N,\gamma,\rho_N),X\rangle=\int_N\gamma(X|_N)+\int_N\rho_N(X|_N^\top),
		\end{gather*}
		where $N\in\Gr_S(M)$, $\gamma\in\Gamma(|\Lambda|_N\otimes\Ann(TN))$, $\rho_N\in\Omega^1(N;|\Lambda|_N)$ such that $(N,\rho_N)\cong(S,\rho)$, and $X\in\mathfrak X_c(M)$.
		Moreover, $J$ identifies certain unions of connected components of $T^*_\reg\Gr_S(M)\times_{\Gr_S(M)}\Gr_{S,\rho}^\deco(M)$ with coadjoint orbits of the full group $\Diff_c(M)$.
	\end{theorem}

\section{The EPDiffvol dual pair for codimension at least two}\label{S:EPDiffvol}

	Suppose $\mu$ is a volume form on $M$. 
	We will denote the Lie algebra of compactly supported exact divergence free vector fields by 
	\[
		\mathfrak X_{c,\ex}(M,\mu):=\bigl\{X\in\mathfrak X_c(M):i_X\mu=d\alpha,\alpha\text{ with compact support}\bigr\}.
	\]
	This is an ideal in the Lie algebra $\mathfrak X_c(M,\mu)$ of compactly supported divergence free vector fields.
	Correspondingly, we let $\Diff_{c,\ex}(M,\mu)$ denote the group of compactly supported exact volume preserving diffeomorphisms, 
	obtained by integrating time dependent vector fields in $\mathfrak X_{c,\ex}(M,\mu)$.
	This is a normal subgroup in the group $\Diff_c(M,\mu)$ of compactly supported volume preserving diffeomorphisms.

\subsection{EPDiffvol dual pair}\label{SS:EPDiffvol}

	Assuming
	\begin{equation}\label{E:codim}
		\dim M-\dim S\geq2
	\end{equation}
	we will show that the dual pair in \eqref{hm} remains a dual pair when the $\Diff_c(M)$ action is restricted to the subgroup $\Diff_{c,\ex}(M,\mu)$.
	Indeed, the following lemma shows that the infinitesimal orbits of these two groups acting on $\mathcal E$ coincide.

	\begin{lemma}\label{L:zetaE}
		If $\Phi\in\mathcal E$ and $X\in\mathfrak X(M)$ then there exists $Y\in\mathfrak X_{c,\ex}(M)$ such that $\zeta^{\mathcal E}_X(\Phi)=\zeta^{\mathcal E}_Y(\Phi)$, 
		where $\zeta^{\mathcal E}$ denotes the infinitesimal $\Diff(M)$ action on $\mathcal E$.
	\end{lemma}

	Recall that the infinitesimal action of $X\in\mathfrak X(M)$ on $\mathcal E=\Emb_\lin(|\Lambda|_S^*,T^*M)$ is
	\[
		\zeta_X^{\mathcal E}(\Phi)=\zeta_X^{T^*M}\circ\Phi,
	\]
	where $\zeta^{T^*M}$ denotes the infinitesimal $\Diff(M)$ action on $T^*M$.
	Thus, Lemma~\ref{L:zetaE} is an immediate consequence of the following

	\begin{lemma}\label{L:zetaT*M}
		Let $N\subseteq M$ be a compact submanifold without boundary and {with} codimension $\dim M-\dim N\geq2$.
		Moreover, suppose $\alpha\in\Gamma(T^*M|_N)$ is a nowhere vanishing section and $X\in\mathfrak X(M)$.
		Then there exists an exact divergence free vector field $Y\in\mathfrak X_{c,\ex}(M,\mu)$ such that $\zeta^{T^*M}_X\circ\alpha=\zeta^{T^*M}_Y\circ\alpha$.
		In particular, $X|_N=Y|_N$. 
	\end{lemma}

	\begin{proof}
		We begin by constructing a vector field $Z\in\mathfrak X(M)$ such that
		\[
		        Z_x=0,\qquad\alpha(D_xZ)=0,\qquad\text{and}\qquad\tr(D_xZ)>0
		\]
		for all $x\in N$, where $D_xZ\in\eend(T_xM)$ denotes the (canonical) derivative of $Z$ at the zero $x$.
		To this end we fix a fiber-wise Euclidean inner product on $TM|_N$ and consider $\tilde Z\in\Gamma(\eend(TM|_N))$ 
		given by $\tilde Z=P_{\ker\alpha}P_{T^\perp N}$,
		where $P_{\ker\alpha}\in\Gamma(\eend(TM|_N))$ denotes the fiber-wise orthogonal projection onto the subbundle $\ker\alpha\subseteq TM|_N$ 
		and $P_{T^\perp N}\in\Gamma(\eend(TM|_N))$ denotes the fiber-wise orthogonal projection onto the orthogonal complement of $TN$ in $TM|_N$.
		It is straightforward to check
		\[
			\rk(T^\perp N)-1\leq\tr(\tilde Z)\leq\rk(T^\perp N).
		\]
		In particular, $\tr(\tilde Z)>0$ in view of the assumption $\dim M-\dim N\geq2$.
		Using a trivialization $T(T^\perp N)\cong(p^{T^\perp N})^*TM|_N$, the section $\tilde Z$ gives rise to a fiber-wise linear vector field on the total space of the normal bundle $T^\perp N$.
		With the help of a tubular neighborhood of $N$ this can be chopped off and glued into $M$ to obtain a vector field $Z$ on $M$ such that $Z_x=0$ and $D_xZ=\tilde Z_x$ for all $x\in N$.
		Hence, by construction, we have $\alpha(D_xZ)=0$ and $\tr(D_xZ)>0$ for each $x\in N$.

		We will next construct a smooth function $f$ on $M$ such that the vector field $\tilde X:=X-fZ$ satisfies
		\begin{equation}\label{E:tX}
			\zeta^{T^*M}_{\tilde X}\circ\alpha=\zeta^{T^*M}_X\circ\alpha\qquad\text{and}\qquad(L_{\tilde X}\mu)|_N=0.
		\end{equation}
		Note that the first equation follows from $(fZ)_x=0$ and $\alpha(D_x(fZ))=0$ for $x\in N$.
		Moreover, we clearly have
		\[
			L_{\tilde X}\mu=L_X\mu-fL_Z\mu-(Z\cdot f)\mu.
		\]
		Restricting to $N$ and using $Z|_N=0$, we obtain
		\[
			(L_{\tilde X}\mu)|_N=(L_X\mu)|_N-f|_N\tr(DZ|_N)\mu.
		\]
		Hence, any function $f$ with restriction
		\[
			f|_N=\frac{(L_X\mu)|_N}{\tr(DZ|_N)\mu}
		\]
		will do the job.

		Proceeding as in \cite[Proposition~2]{HV04}, we use the relative Poincar\'e lemma, see \cite[\S43.10]{KM97} for instance, 
		to conclude that there exist an open neighborhood $U$ of $N$ in $M$ and a differential form $\lambda$ of degree $\dim M-1$ on $U$ 
		such that $d\lambda=(L_{\tilde X}\mu)|_U$ and $(j^1\lambda)|_N=0$ where $(j^1\lambda)|_N$ denotes the 1-jet of $\lambda$ along $N$.
		The equation $i_Z\mu=\lambda$ defines a vector field $Z$ on $U$ such that $L_Z\mu=L_{\tilde X}\mu$ on $U$ and $(j^1Z)|_N=0$.
		For the vector field $\tilde Y:=\tilde X|_U-Z$ on $U$ we obtain
		\begin{equation}\label{E:tY}
			(j^1\tilde Y)|_N=(j^1\tilde X)|_N\qquad\text{and}\qquad L_{\tilde Y}\mu=0.
		\end{equation}
		Since the codimension of $N$ is at least 2, there exists a vector field $Y\in\mathfrak X_{c,\ex}(M,\mu)$ which coincides with $\tilde Y$ on a neighborhood of $N$.
		Combining this with \eqref{E:tX} and \eqref{E:tY}, we obtain $\zeta^{T^*M}_X\circ\alpha=\zeta^{T^*M}_Y\circ\alpha$.
	\end{proof}

	Combining these observations with the EPDiff dual pair in Theorem~\ref{epdiff}, we obtain the following result.

	\begin{theorem}[EPDiffvol dual pair for codimension at least two]\label{T:dual.pair}
		Let $\mu$ be a volume form on a smooth manifold $M$, and suppose $S$ is a closed manifold such that $\dim M-\dim S\geq2$.
		Then the weak symplectic Fr\'echet manifold $\mathcal E=T^*_{\reg,\times}\Emb(S,M)$ in \eqref{E:varE}, together with the commuting actions of $\Diff_{c,ex}(M,\mu)$ and $\Diff(S)$, 
		and the equivariant moment maps associated with the canonical invariant 1-form $\theta^{\mathcal E}$, 
		\begin{equation}\label{E:dp.vol}
			\mathfrak X_{c,\ex}(M,\mu)^*\xleftarrow{J_L^{\mathcal E}}\mathcal E\xrightarrow{J_R^{\mathcal E}}\Omega^1(S;|\Lambda|_S)\subseteq\mathfrak X(S)^*,
		\end{equation}
		constitute a symplectic dual pair with mutually symplectic orthogonal orbits.
		More explicitly, the moment maps are
		\begin{equation}\label{E:moma.vol}
			\langle J^{\mathcal E}_L(\varphi,\alpha),X\rangle=\int_S\alpha(X\circ\varphi)
			\qquad\text{and}\qquad
			J^{\mathcal E}_R(\varphi,\alpha)=\varphi^*\alpha,
		\end{equation}
		where $X\in\mathfrak X_{c,\ex}(M,\mu)$ and $(\varphi,\alpha)\in\mathcal E$.
		An analogous statement remains true if $\Diff_{c,\ex}(M,\mu)$ and $\mathfrak X_{c,\ex}(M,\mu)$ are replaced with $\Diff_c(M,\mu)$ and $\mathfrak X_c(M,\mu)$, respectively.
	\end{theorem}
	
	The left moment map $J_L^{\mathcal E}$ of the EPDiffvol dual pair in \eqref{E:dp.vol} descends to a $\Diff(M,\mu)$ equivariant injective map
	\begin{equation}\label{E:barJL.mu}
		\bar J_L^{\mathcal E}\colon\mathcal G\to\mathfrak X_{c,\ex}(M,\mu)^*,\quad\langle\bar J_L^{\mathcal E}(N,\gamma),X\rangle
		=\int_N\gamma(X|_N),
	\end{equation}
	where $(N,\gamma)\in\mathcal G$ and $X\in\mathfrak X_{c,\ex}(M,\mu)$.
	Clearly, this map coincides with the map obtained from \eqref{E:barJL} by restricting to the subalgebra $\mathfrak X_{c,\ex}(M,\mu)$.
	The injectivity of \eqref{E:barJL.mu} follows from the injectivity of \eqref{E:barJL} and the last statement in Lemma~\ref{L:zetaT*M}.

\subsection{Reduction at zero}\label{SS:32}

	In this section we discuss symplectic reduction at the level zero for the $\Diff(S)$ action of the EPDiff\-vol dual pair described in Theorem~\ref{T:dual.pair}.
	We will see that each connected component of $\mathcal G^\iso=T^*_{\reg,\times}\Gr_S(M)$, cf.~\eqref{E:Giso}, is a coadjoint orbit of $\Diff_{c,\ex}(M,\mu)$.

	The map $\bar J_L^{\mathcal E}$ in~\eqref{E:barJL.mu} restricts to a $\Diff(M,\mu)$ equivariant injective map
	\begin{equation}\label{E:barJLiso.mu}
		\bar J^\iso_L\colon\mathcal G^\iso\to\mathfrak X_{c,\ex}(M,\mu)^*
	\end{equation}
	which coincides with the map obtained from \eqref{E:barJLiso} by restricting to the subalgebra $\mathfrak X_{c,\ex}(M,\mu)$.
	According to Theorem~\ref{T:dual.pair} the $\Diff_{c,\ex}(M,\mu)$ action on $\mathcal G^\iso$ is infinitesimally transitive.
	Proceeding as in the proof of \cite[Proposition~2]{HV04} we see that the $\Diff_{c,\ex}(M,\mu)$ action on $\mathcal G^\iso$ is locally transitive.
	Hence, \eqref{E:barJLiso.mu} identifies each connected component of $\mathcal G^\iso$ with a coadjoint orbit of $\Diff_{c,\ex}(M,\mu)$.
	Moreover, $(\bar J^\iso_L)^*\omega^\KKS=\omega^\iso$ where $\omega^\KKS$ denotes the Kirillov--Kostant--Souriau symplectic form on the coadjoint orbit.

	Summarizing these observation we obtain the following result analogous to Theorem \ref{T:t}.

	\begin{corollary}
		Let $\mu$ be a volume form on a smooth manifold $M$, and suppose $S$ is a closed manifold such that $\dim M-\dim S\geq2$.
		Then each connected component of $T^*_{\reg,\times}\Gr_S(M)$, endowed with its canonical symplectic form and the cotangent lifted $\Diff_{c,\ex}(M,\mu)$ action, 
		is symplectomorphic to a coadjoint orbit of $\Diff_{c,\ex}(M,\mu)$ via the equivariant moment map 
		\[
			J\colon T^*_{\reg,\times}\Gr_S(M)\to \mathfrak X_{c,\ex}(M,\mu)^*,\quad
			\langle J(N,\gamma),X\rangle=\int_N\gamma(X|_N),
		\]
		where $\gamma\in\Gamma(|\Lambda|_N\otimes\Ann(TN))$ nowhere zero and $X\in\mathfrak X_{c,\ex}(M,\mu)$.
		An analogous statement remains true if $\Diff_{c,\ex}(M,\mu)$ and $\mathfrak X_{c,\ex}(M,\mu)$ are replaced with $\Diff_c(M,\mu)$ and $\mathfrak X_c(M,\mu)$, respectively.
	\end{corollary}

\subsection{Reduction at nonzero levels}\label{SS:33}
	
	Assuming $\dim S=1$ and $\dim M\geq3$ we will now discuss reduction for the $\Diff(S)$ action of the EPDiffvol dual pair in Theorem~\ref{T:dual.pair} 
	at a level $\rho\in\Omega^1(S;|\Lambda|_S)\subseteq\mathfrak X(S)^*$ which is nowhere zero.

	Recall from Section~\ref{SS:EPDiff.nonzero} that in this situation the reduced symplectic space of augmented closed curves,
	\[
		\mathcal G_\rho=T^*_\reg\Gr_S(M)\times_{\Gr_S(M)}\Gr_{S,\rho}^\deco(M),
	\]
	admits a natural Fr\'echet manifold structure.
	The map $\bar J_L^{\mathcal E}$ in~\eqref{E:barJL.mu} restricts to a $\Diff(M,\mu)$ equivariant injective map
	\begin{equation}\label{E:barJLrho.mu}
		\bar J^\rho_L\colon\mathcal G_\rho\to\mathfrak X_{c,\ex}(M,\mu)^*
	\end{equation}
	which coincides with the map obtained from \eqref{E:barJLrho} by restricting to the subalgebra $\mathfrak X_{c,\ex}(M,\mu)$.
	According to Theorem~\ref{T:dual.pair} the $\Diff_{c,\ex}(M,\mu)$ action on $\mathcal G_\rho$ is infinitesimally transitive.
	Proceeding as in the proof of \cite[Proposition~2]{HV04} we see that the $\Diff_{c,\ex}(M,\mu)$ action on $\mathcal G_\rho$ is locally transitive.
	Hence, \eqref{E:barJLrho.mu} identifies each connected component of $\mathcal G_\rho$ with a coadjoint orbit of $\Diff_{c,\ex}(M,\mu)$,
	and we obtain a result analogous to Theorem \ref{T:EPDiff.red1dim}.

	\begin{corollary}\label{C:Riemann}
		In the situation described above, each connected component of $T^*_\reg\Gr_S(M)\times_{\Gr_S(M)}\Gr_{S,\rho}^\deco(M)$, 
		endowed with the symplectic form in \eqref{E:orr} and the natural $\Diff_{c,\ex}(M,\mu)$ action,
		is symplectomorphic to a coadjoint orbit of $\Diff_{c,\ex}(M,\mu)$ via the equivariant moment map 
		\begin{gather*}
			J\colon T^*_\reg\Gr_S(M)\times_{\Gr_S(M)}\Gr_{S,\rho}^\deco(M)\to\mathfrak X_{c,\ex}(M,\mu)^*
			\\
			\langle J(N,\gamma,\rho_N),X\rangle=\int_N\gamma(X|_N)+\int_N\rho_N(X|_N^\top),
		\end{gather*}
		where $N\in\Gr_S(M)$, $\gamma\in\Gamma(|\Lambda|_N\otimes\Ann(TN))$, $\rho_N\in\Omega^1(N;|\Lambda|_N)$ such that $(N,\rho_N)\cong(S,\rho)$, and $X\in\mathfrak X_{c,\ex}(M,\mu)$.
		An analogous statement remains true if $\Diff_{c,\ex}(M,\mu)$ and $\mathfrak X_{c,\ex}(M,\mu)$ are replaced with $\Diff_c(M,\mu)$ and $\mathfrak X_c(M,\mu)$, respectively.
	\end{corollary}

\section{Foliations on codimension one nonlinear Grassmannians}\label{S:foliations}

	For the remaining part of the paper we will assume
	\begin{equation}\label{E:codim1}
		\dim M-\dim S=1.
	\end{equation}
	In this situation the volume form gives rise to integrable distributions of finite codimension on $\Gr_S(M)$,
	which are analogous to Weinstein's isodrastic foliation in the symplectic case \cite{W90,L09}.
	The corresponding leaves on $\Emb(S,M)$ will be essential in Section \ref{S:codim.one},
	where we construct dual pairs for the volume preserving diffeomorphism group in the codimension one case.
	These isodrastic leaves for volume forms have been useful in the description of coadjoint orbits
	of the volume preserving diffeomorphism group consisting of vortex sheets, i.e., codimension one submanifolds decorated with closed 1-forms \cite{GBV23}.

\subsection{Isodrastic foliation}\label{iso}

	For $N\in\Gr_S(M)$ we let $\mathcal D_{N,\ex}$ denote the subspace consisting of all tangent vectors $X\in T_N\Gr_S(M)=\Gamma(TM|_N/TN)$ such that $i_X\mu$ becomes exact on $N$.
	This defines a distribution $\mathcal D_\ex$ on $\Gr_S(M)$ that we call the \emph{isodrastic distribution} in analogy to Weinstein's isodrastic distribution for symplectic manifolds.
	Clearly, $\mathcal D_\ex$ is invariant under $\Diff(M,\mu)$.

	Contraction with the volume form followed by restriction to $TN$ provides an isomorphism of line bundles,
	\begin{equation}\label{E:mu}
		\mu_N\colon TM|_N/TN\to\Lambda^{\dim N}T^*N.
	\end{equation}
	Using the canonical isomorphism $|\Lambda|_N=\mathfrak o_N\otimes\Lambda^{\dim N}T^*N$, 
	where $\mathfrak o_N$ denotes the orientation bundle over $N$, we regard this as an isomorphism of line bundles
	\begin{equation}\label{E:mu.Gr}
		\mu_N\colon\mathfrak o_N\to|\Lambda|_N\otimes\Ann(TN).
	\end{equation}
	Hence, every regular covector $\xi\in T^*_{N,\reg}\Gr_S(M):=\Gamma(|\Lambda|_N\otimes\Ann(TN))$ can be written in the form 
	$\xi=f\mu_N$ for a unique $f\in\Gamma(\mathfrak o_N)$ which we denote by $f=\xi/\mu_N$.
	Clearly,
	\[
		\langle\xi,X\rangle=\int_Nfi_X\mu
	\]
	for $X\in T_N\Gr_S(M)=\Gamma(TM|_N/TN)$. 
	Using Stokes' theorem we conclude
	
	\begin{lemma}\label{L:cove}
		A regular covector $\xi\in T^*_{N,\reg}\Gr_S(M)$ vanishes on $\mathcal D_{N,\ex}$ if and only if $f=\xi/\mu_N$ is a parallel section of $\mathfrak o_N$. 
		\footnote{The orientation bundle $\mathfrak o_N$ is a line bundle whose structure group is reduced to ${\rm O}(1)$, and which therefore comes equipped with a canonical flat connection.}
		Hence, the isomorphism of line bundles in \eqref{E:mu.Gr} provides an isomorphism of finite dimensional vector spaces,
		\begin{equation}\label{E:mu.Gr2}
			\mu_N:H^0(N;\mathfrak o_N)\to\Ann_\reg(\mathcal D_{N,\ex}),
		\end{equation}
		where $\Ann_\reg(\mathcal D_{N,\ex}):=\{\xi\in T^*_{N,\reg}\Gr_S(M):\xi|_{\mathcal D_{N,\ex}}=0\}$ denotes the regular annihilator.
		In particular, the isodrastic distribution has finite codimension
		\begin{multline}\label{E:codim.Fcex}
			\codim\mathcal D_{\ex}
			=\dim H^0(S;\mathfrak o_S)
			\\
			=\textrm{number of orientable connected components of $S$.}
		\end{multline}
	\end{lemma}
	 
	\begin{lemma}\label{L:Fcex.int}
		The isodrastic distribution $\mathcal D_\ex$ on $\Gr_S(M)$ is integrable.
		Its leaves coincide with the orbits of $\Diff_{c,\ex}(M,\mu)$. 
	\end{lemma}

	\begin{proof}
		W.l.o.g.\ we may assume $S$ to be connected and orientable.
		Suppose $N\in\Gr_S(M)$.
		Using a tubular neighborhood of $N$ and \cite[Section~4]{M65},	
		we may assume $M=N\times\mathbb R$ and $\mu=dt\wedge p^*\nu$ where $p\colon M\to N$ and $t\colon M\to\mathbb R$ denote the natural projections, and $\nu$ is a volume form on $N$.
		For $f\in C^\infty(N;\mathbb R)$ we let $N_f:=\{(x,f(x)):x\in N\}$ denote its graph in $M$.
		This provides a standard chart for the smooth structure on $\Gr_S(M)$ centered at $N$.
		We let $\mathcal U$ denote the open neighborhood of $N$ in $\Gr_S(M)$ corresponding to $C^\infty(N,\mathbb R)$ via this chart.
		Consider
		\begin{equation}\label{E:V.def}
			V\colon\mathcal U\to\mathbb R,\qquad
			V(N_f):=\int_Nf\nu.
		\end{equation}
		As the integral is a bounded linear functional, the level sets of $V$ provide a codimension one foliation on $\mathcal U$.
		Clearly,
		\[
			dV(X)=\int_N\lambda\nu=\int_{N_f}i_X\mu,
		\]
		where $X\in T_{N_f}\Gr_S(M)=\Gamma(TM|_{N_f}/TN_f)$ is identified with $\lambda\in C^\infty(N,\mathbb R)$, the pull back of $dt(X)\in C^\infty(N_f,\mathbb R)$ to $N$ via $f$.
		We conclude that the level sets of $V$ are leaves of $\mathcal D_\ex|_{\mathcal U}$.
		Note here, that the level sets of $V$ are connected, by convexity.

		Clearly, the orbits of $\Diff_{c,\ex}(M,\mu)$ are tangential to $\mathcal D_\ex$.
		It remains to show that $\Diff_{c,\ex}(M,\mu)$ acts locally transitively on the leaves of $\mathcal D_\ex$.
		To this end, suppose $N_f$ and $N$ are in the same leaf of $\mathcal D_\ex|_{\mathcal U}$.
		As the latter foliation is given by the level sets of $V$ in \eqref{E:V.def}, we have $\int_Nf\nu=0$.
		Note that $g(x,t):=(x,t+f(x))$ defines a volume preserving diffeomorphism of $M=N\times\mathbb R$ mapping $N$ to $N_f$.
		This diffeomorphism is the flow at time one of the vector field $Y(x,t)=f(x)\partial_t$.
		This vector field is divergence-free, for we have $i_Y\mu=p^*(f\nu)$ and thus $di_Y\mu=0$.
		As $\int_Nf\nu=0$, there exists $\alpha\in\Omega^*(N)$ such that $f\nu=d\alpha$.
		Choose $\lambda\in C^\infty_c(\mathbb R,\mathbb R)$ such that $\lambda=1$ on the interval $[-a,a]$ where $a:=\max_{x\in N}|f(x)|$.
		Define $Z\in\mathfrak X_{c,\ex}(M)$ by $i_Z\mu:=d((t^*\lambda)(p^*\alpha))$ and let $h\in\Diff_{c,\ex}(M)$ denote the flow at time one of $Z$.
		By construction, $Y$ and $Z$ coincide on $N\times[-a,a]$.
		Hence, $h(N)=g(N)=N_f$.
	\end{proof}
	
	\begin{remark}\label{R:isodrast.cc}
		The isodrastic distribution constrains the connected components of $N\in\Gr_S(M)$ independently.
		To formulate this more precisely, let $S_1,\dotsc,S_k$ denote the connected components of $S$.
		Moreover, let $U$ denote the open subset in $\prod_{i=1}^k\Gr_{S_i}(M)$ consisting of all tuples $(N_1,\dotsc,N_k)$ such that the submanifolds $N_i$ are mutually disjoint.
		The map $U\to\Gr_S(M)$ provided by the (disjoint) union is a principal covering with structure group $\{\sigma\in\mathfrak S_k:\text{$S_i\cong S_{\sigma(i)}$ for $i=1,\dotsc,k$}\}$.
		This covering map intertwines the product of the isodrastic distributions restricted to $U\subseteq\prod_{i=1}^k\Gr_{S_i}(M)$ with the isodrastic distribution on $\Gr_S(M)$.
	\end{remark}

	The pull back of the isodrastic distribution on $\Gr_S(M)$ is a distribution on $\Emb(S,M)$ which is invariant under $\Diff(M,\mu)$ and $\Diff(S)$.
	We also denote this distribution by $\mathcal D_\ex$ and we call it \emph{isodrastic distribution} too.
		
	\begin{lemma}\label{L:Fcex.int.Emb}
		The isodrastic distribution $\mathcal D_\ex$ on $\Emb(S,M)$ is integrable with (finite) codimension as in \eqref{E:codim.Fcex}.
		Its leaves coincide with the orbits of $\Diff_{c,\ex}(M,\mu)$.
		For $\varphi\in\Emb(S,M)$ the isomorphism of line bundles provided by $\mu$,
		\begin{equation}\label{E:mu.iso}
			\mu_\varphi\colon\mathfrak o_S\to|\Lambda|_S\otimes\varphi^*\Ann(T\varphi(S))\subseteq|\Lambda|_S\otimes\varphi^*T^*M,
		\end{equation}
		cf.~\eqref{E:mu.Gr}, induces an isomorphism of finite dimensional vector spaces,
		\begin{equation}\label{E:ann}
			\mu_\varphi\colon H^0(S;\mathfrak o_S)\to\Ann_\reg(\mathcal D_{\varphi,\ex}),
		\end{equation}
		where $\Ann_\reg(\mathcal D_{\varphi,\ex})=\{\alpha\in T^*_{\varphi,\reg}\Emb(S,M):\alpha|_{\mathcal D_{\varphi,\ex}}=0\}$.
	\end{lemma}

	\begin{proof}
		This follows from Lemmas~\ref{L:cove} and \ref{L:Fcex.int} using the fact that every diffeomorphism in $\Diff(N)_0$ can be extended to a diffeomorphism in $\Diff_{c,\ex}(M,\mu)$.
	\end{proof}
	
	The foliations on $\Gr_S(M)$ and $\Emb(S,M)$ integrating the isodrastic distributions will be referred to as \emph{isodrastic foliations.} 
	Its leaves, i.e., the $\Diff_{c,\ex}(M,\mu)$ orbits, are called \emph{isodrastic leaves.}

	\begin{remark}\label{R:open}
		If $M$ is open, then there exists a form $\vartheta$ such that $\mu=d\vartheta$.
		Every such form provides a map
		\[
			F_\vartheta\colon\Emb(S,M)\to H^{\dim M-1}(S),\qquad F_\vartheta(\varphi):=[\varphi^*\vartheta].
		\]
		One readily checks $(dF_\vartheta)(X)=[\varphi^*i_X\mu]$ for $X\in T_\varphi\Emb(S,M)$.
		Hence, $F_\vartheta$ is a submersion and $\ker(dF_\vartheta)=\mathcal D_\ex$.
		Therefore, the connected components of $F_\vartheta$-level sets are isodrastic leaves.
		In particular, each connected component of $F_\vartheta^{-1}(0)$, the set of $\vartheta$-exact embeddings, is an isodrastic leaf.
		By equivariance, $F_\vartheta$ provides a section of the associated flat vector bundle
		\[
			\Emb(S,M)\times_{\Diff(S)} H^{\dim M-1}(S)\to\Gr_S(M),
		\]
		a homologically decorated nonlinear Grassmannian.
		Each connected component of the preimage of a horizontal leaf is an isodrastic leaf.
		In particular, the connected components of the preimage of the zero section, i.e., the set of all $\vartheta$-exact submanifolds $N\in\Gr_S(M)$, are isodrastic leaves.
	\end{remark}

\subsection{Regular cotangent bundle of isodrastic leaves}\label{SS:T*regL}

	Let $L$ be the preimage under $\Emb(S,M)\to\Gr_S(M)$ of an isodrastic leaf in $\Gr_S(M)$.
	Note that $L$ will not be connected, in general.
	Its connected components are isodrastic leaves in $\Emb(S,M)$.
	We have commuting actions of $\Diff_{c,\ex}(M,\mu)$ and $\Diff(S)$ on $L$.

	For $\varphi\in L$ we denote the (finite dimensional) annihilator of $T_\varphi L=\mathcal D_{\varphi,\ex}$ by
	\[
		\Ann_\reg(T_\varphi L):=\bigl\{\alpha\in T^*_{\varphi,\reg}\Emb(S,M):\alpha|_{T_\varphi L}=0\bigr\}=\Ann_\reg(\mathcal D_{\varphi,\ex}),
	\]
	and define
	\[
		T_{\varphi,\reg}^*L:=T_{\varphi,\reg}^*\Emb(S,M)/\Ann_\reg(T_\varphi L).
	\]
	The annihilators provide a vector subbundle $\Ann_\reg(TL):=\bigsqcup_{\varphi\in L}\Ann_\reg(T_\varphi L)$ of finite rank in $(T^*_\reg\Emb(S,M))|_L$.
	We define the regular cotangent bundle of $L$ by
	\begin{equation}\label{E:T*regL}
		T_\reg^*L:=\bigl(T^*_\reg\Emb(S,M)\bigr)\big|_L/\Ann_\reg(TL).
	\end{equation}

	From Lemma~\ref{L:Fcex.int.Emb} we obtain
	\[
		\codim L=\rank\Ann_\reg(TL)=\dim H^0(S;\mathfrak o_S)
	\]
	as well as the more explicit description
	\begin{equation}\label{explit}
		T_\reg^*L=\bigl\{(\varphi,[\alpha]):\varphi\in L,\alpha\in\Gamma(|\Lambda|_S\otimes\varphi^*T^*M)\bigr\}
	\end{equation}
	where $[\alpha_1]=[\alpha_2]$ iff $\alpha_2-\alpha_1=f\mu_\varphi$ with $f\in\Gamma_\locconst(\mathfrak o_S)=H^0(S;\mathfrak o_S)$.

	Some natural maps will be denoted as indicated in the following diagram:
	\begin{equation}\label{E:diag}
		\vcenter{\xymatrix{
			T^*_\reg L\ar@/_2ex/[dr]_-p&\bigl(T^*_\reg\Emb(S,M)\bigr)\big|_L\ar[l]_-q\ar@{^(->}[r]^-\iota\ar[d]^p&T^*_\reg\Emb(S,M)\ar[d]^p
			\\
			&L\ar@{^(->}[r]&\Emb(S,M)
		}}
	\end{equation}
	We equip $T^*_\reg L$ with the unique structure of a Fr\'echet manifold such that $q$ becomes a smooth vector bundle (with finite dimensional fibers).
	The actions of $\Diff_{c,\ex}(M,\mu)$ and $\Diff(S)$ on $T^*_\reg\Emb(S,M)$ restrict and descend naturally 
	to commuting smooth actions on each manifold in the diagram such that $q$, $\iota$, and $p$ are all equivariant.
	The canonical 1-form on $T^*_\reg L$ can be characterized by
	\begin{equation}\label{E:thetaL}
		q^*\theta^{T^*_\reg L}=\iota^*\theta^{T^*_\reg\Emb(S,M)}.
	\end{equation}
	Clearly, $\theta^{T^*_\reg L}$ is invariant under both group actions.
	Moreover, $\omega^{T^*_\reg L}:=d\theta^{T^*_\reg L}$ is a closed 2-form on $T^*_\reg L$ which is also invariant under these group actions.
	Equivalently this 2-form can be characterized by
	\begin{equation}\label{E:omegaL}
		q^*\omega^{T^*_\reg L}=\iota^*\omega^{T^*_\reg\Emb(S,M)}.
	\end{equation}

	\begin{lemma}
		The closed 2-form $\omega^{T^*_\reg L}$ is a weak symplectic form on $T^*_\reg L$. 
	\end{lemma}

	\begin{proof}
		Since $\omega^{T^*_\reg\Emb(S,M)}$ is weakly nondegenerate it induces, together with $Tp$, a nondegenerate pairing of finite dimensional vector spaces
		\begin{equation}\label{E:pairing}
			\ker\bigl(\iota^*\omega_{(\varphi,\alpha)}^{T^*_\reg\Emb(S,M)}\bigr)\times\frac{(T_\varphi\Emb(S,M))|_L}{T_\varphi L}\to\mathbb R.
		\end{equation}
		Clearly, $\ker(Tq)\subseteq\ker\bigl(\iota^*\omega^{T^*_\reg\Emb(S,M)}\bigr)$ in view of \eqref{E:omegaL}.
		Using the nondegeneracy of the pairing \eqref{E:pairing} we conclude
		\begin{equation}\label{E:kerTq}
			\ker(Tq)=\ker\bigl(\iota^*\omega^{T^*_\reg\Emb(S,M)}\bigr)
		\end{equation}
		as both vector bundles have the same finite dimensional rank.
		Hence, $\omega^{T^*_\reg L}$ is weakly nondegenerate.
	\end{proof}

\subsection{Isovolume foliation}\label{SS:isovol}

	We obtain a variation of the isodrastic distribution on $\Gr_S(M)$ by considering all tangent vectors $X\in T_N\Gr_S(M)=\Gamma(TM|_N/TN)$ 
	such that the cohomology class $[i_X\mu]\in H^{\dim N}(N;\mathbb R)$ is in the image of the natural homomorphism
	\begin{equation}\label{E:iotaN*}
		\iota_N^*\colon H_c^{\dim N}(M;\mathbb R)\to H^{\dim N}(N;\mathbb R).
	\end{equation}
	In view of Proposition~\ref{P:below} below, we will call this the \emph{isovolume distribution} and we will denote it by $\mathcal D_\mu$.
	Clearly, this distribution is invariant under $\Diff(M,\mu)$.
	Obviously, the codimension of $\mathcal D_\mu$ is at most the codimension of $\mathcal D_\ex$.
	In general, the codimension of $\mathcal D_\mu$ will depend on the connected component of $\Gr_S(M)$. 
	More precisely we have the following analogue of Lemma~\ref{L:Fcex.int}:
	
	\begin{lemma}\label{L:cmu}
		The distribution $\mathcal D_\mu$ on $\Gr_S(M)$ is integrable.
		Its leaves coincide with the orbits of $\Diff_{c}(M,\mu)_0$. 
		The (finite) codimension at $N\in\Gr_S(M)$ is
		\begin{equation}\label{E:codim.Fcmu}
			\codim_N\mathcal D_\mu=\dim\coker\bigl(H_c^{\dim N}(M;\mathbb R)\to H^{\dim N}(N;\mathbb R)\bigr).
		\end{equation}
	\end{lemma}

	The corresponding foliation on $\Gr_S(M)$ will be called \emph{isovolume foliation}.
	
	\begin{remark}\label{R:null}
		If every orientable connected component of $N$ is null-homologous in $M$, then the distributions $\mathcal D_\mu$ and $\mathcal D_\ex$ coincide at $N$.
		Indeed, the homomorphism in~\eqref{E:iotaN*} always factors through $H_c^{\dim N}(M;\mathbb R)\to H^{\dim N}(M;\mathbb R)$, giving rise to 
		the homomorphism dual to $(\iota_N)_*\colon H_{\dim N}(N;\mathbb R)\to H_{\dim N}(M;\mathbb R)$, which vanishes by the assumption. 
		The converse is not true in general. 
		To see this, consider the cylinder $M=S^1\times\mathbb R$ and one of its meridians $N=S^1\times\{y\}$. 
		The distributions $\mathcal D_\mu$ and $\mathcal D_\ex$ coincide at $N$, even though the orientable submanifold $N$ is not null-homologous.
	\end{remark}

	\begin{proposition}\label{P:below}
		Let $N_t\in\Gr_S(M)$ be a smooth curve, $t\in[0,1]$.
	
		(a) If $N_t$ is tangent to the isovolume distribution $\mathcal D_\mu$, then there exists a compact submanifold with boundary, $K\subseteq M$, 
		whose interior contains $N_t$ for all $t$, and such that the volume of each connected component of $K\setminus N_t$ is constant in $t$.
	
		(b) Suppose $K\subseteq M$ is a compact submanifold with boundary whose interior contains $N_t$ for all $t$.
		If, moreover, the volume of each connected component of $K\setminus N_t$ is independent of $t$, then $N_t$ is tangent to the isovolume distribution $\mathcal D_\mu$.
	\end{proposition}
	
	\begin{proof}
		(a) Suppose $N_t$ is tangent to $\mathcal D_\mu$.
		By Lemma~\ref{L:cmu} there exists an isotopy $f_t$ in $\Diff_c(M,\mu)$ such that $f_0=\id$ and $f_t(N_0)=N_t$.
		There exists a compact submanifold with boundary $K\subseteq M$ such that $f_t$ is supported in $K$ and such that $N_t$ is contained in the interior of $K$ for all $t$. 
		In particular, $K\setminus N_t=f_t(K\setminus N_0)$. 
		As $f_t$ is volume preserving, the volume of each connected component of $K\setminus N_t$ is constant in $t$.
	
		(b) 
		By the isotopy extension theorem \cite[Theorem~8.1.3]{H76} there exists an isotopy $f_t$ in $\Diff_c(M)$ such that $f_0=\id$ and $f_t(N_0)=N_t$ for all $t$.
		This isotopy may be chosen with support contained in $\mathring K$, the interior of $K$. 
		In particular, $f_t^*\mu-\mu$ is supported in $\mathring K$.
		According to the hypothesis, the integral of this form over each connected component of $K\setminus N_0$ vanishes.
		Via Poincar\'e duality, 
		\[
			H_c^{\dim M}(\mathring K,N_0;\mathbb R)^*=H^0(K\setminus N_0;\mathbb R),
		\]
		this implies that $f_t^*\mu-\mu$ represents the trivial class in $H^{\dim M}_c(\mathring K,N_0;\mathbb R)$.
		We conclude that there exist $\alpha_t\in\Omega_c^{\dim(M)-1}(\mathring K)$ such that $f_t^*\mu-\mu=d\alpha_t$ and $\iota_{N_0}^*\alpha_t=0$.
		Moreover, the forms $\alpha_t$ may be chosen to depend smoothly on $t$. 
		Extending $\alpha_t$ by zero to all of $M$, the equation $f_t^*\mu-\mu=d\alpha_t$ remains true on $M$.
		Proceeding as in \cite[Section~4]{M65}, we consider the time dependent vector field $Y_t$ uniquely characterized by $i_{Y_t}f_t^*\mu=-\frac\partial{\partial t}\alpha_t$. 
		In particular, $Y_t$ is supported in $K$ and tangent to $N_0$.
		Hence, the isotopy $g_t\in\Diff_c(M)$ integrating $Y_t$ and starting at $g_0=\id$ preserves $N_0$.
		By construction, 
		\[
			\tfrac\partial{\partial t}(g_t^*f_t^*\mu)
			=g_t^*\bigl(L_{Y_t}f_t^*\mu+\tfrac\partial{\partial t}f_t^*\mu\bigr)
			=g_t^*\bigl(-d\tfrac\partial{\partial t}\alpha_t+\tfrac\partial{\partial t}d\alpha_t\bigr)
			=0.
		\]
		Hence $(f_t\circ g_t)^*\mu=\mu$ and $f_t\circ g_t\in\Diff_c(M,\mu)_0$.
		Using Lemma~\ref{L:cmu} we conclude that $N_t=(f_t\circ g_t)(N_0)$ is tangential to $\mathcal D_\mu$.
	\end{proof}

	The pull back of the distribution $\mathcal D_\mu$ to $\Emb(S,M)$ will also be called \emph{isovolume distribution} and will be denoted by the same symbol.
	This distribution is invariant under $\Diff(M,\mu)$ and $\Diff(S)$.
	It has finite codimension and it is integrable with $\Diff_c(M,\mu)_0$ orbits as leaves.
	
	We end this section with some examples illustrating the isodrastic and isovolume foliations on $\Gr_S(M)$.

	\begin{example}[Curves in the plane]\label{E:R2.conn}
		We consider $M=\mathbb R^2$ with the standard volume form $\mu=dxdy$ and $S=S^1$.
		In this case the isodrastic foliation and the isovolume foliation on $\Gr_S(M)$ coincide according to Remark~\ref{R:null}, i.e., $\mathcal D_\ex=\mathcal D_\mu$.
		Two elements of $\Gr_S(M)$ which lie in the same isodrastic leaf must enclose the same area by Proposition~\ref{P:below}(a).
		Using part (b) of this proposition one readily shows that two curves enclosing the same area do indeed lie in the same isodrastic leaf.
		If $S$ has more connected components, the situation gets more involved as illustrated in Figure~\ref{F:R2}.
		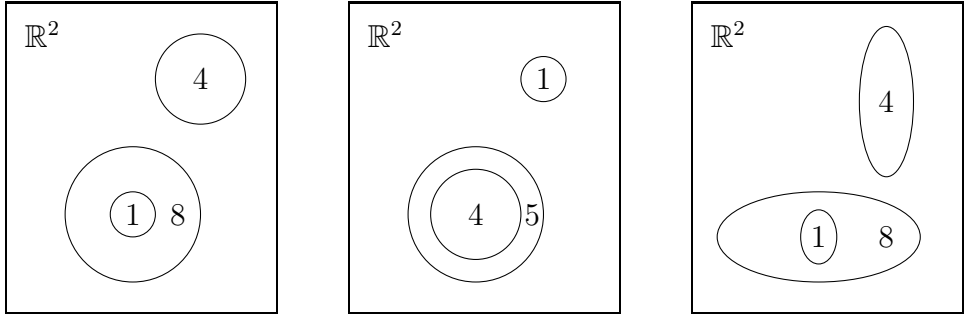
\begin{figure}
			\hfill
			\boxed{\begin{tikzpicture}[scale=.3]
				\draw (0,0) circle (1);
				\draw (0,0) node {$1$};
				\draw (0,0) circle (3);
				\draw (2,0) node {$8$};
				\draw (3,6) circle (2);
				\draw (3,6) node {$4$};
				\draw (-4,8) node {$\mathbb R^2$};
				\draw (6,-4);
			\end{tikzpicture}}
			\hfill
			\boxed{	\begin{tikzpicture}[scale=.3]
				\draw (0,0) circle (2);
				\draw (0,0) node {$4$};
				\draw (0,0) circle (3);
				\draw (2.5,0) node {$5$};
				\draw (3,6) circle (1);
				\draw (3,6) node {$1$};
				\draw (-4,8) node {$\mathbb R^2$};
				\draw (6,-4);
			\end{tikzpicture}}
			\hfill
			\boxed{\begin{tikzpicture}[scale=.3]
				\draw (0,-1) ellipse (0.8 and 1.2);
				\draw (0,-1) node {$1$};
				\draw (0,-1) ellipse (4.5 and 2);
				\draw (3,-1) node {$8$};
				\draw (3,5) ellipse (1.2 and 3.333);
				\draw (3,5) node {$4$};
				\draw (-4,8) node {$\mathbb R^2$};
				\draw (6,-4);
			\end{tikzpicture}}
			\hfill\
			\caption{
				Three elements in the same connected component of $\Gr_S(M)$ for $M=\mathbb R^2$ and $S=S^1\sqcup S^1\sqcup S^1$. 
				Numbers indicate the areas of the corresponding regions.
				The submanifolds indicated on the left and  on the right are in the same isodrastic leaf, the one indicated in the middle lies in a different leaf.
				Note that each submanifold consists of three connected components enclosing areas 1, 4 and 9, individually.
			}\label{F:R2}
		\end{figure}
	\end{example}
	
	\begin{example}[Connected curves in the 2-torus]\label{E:T2.conn}
		We consider the 2-torus $M=T^2=\mathbb R^2/\mathbb Z^2=S^1\times S^1$ with the standard (invariant) volume form $\mu=dxdy$ and $S=S^1$.
		In this case the isodrastic foliation has codimension one, see~\eqref{E:codim.Fcex}.
		If $N\in\Gr_S(M)$ is close to the meridian $N_0=S^1\times\{y_0\}$, then it is of the form $N=\{(x,f(x)):x\in S^1\}$ for some function $f\colon S^1\to\mathbb R$.
		Such an $N$ lies in the isodrastic leaf through $N_0$ iff $\int_{S^1}f(x)dx=0$, see Remark~\ref{R:open}.
		The isovolume foliation, on the other hand, has codimension zero at $N_0$, see \eqref{E:codim.Fcmu}.
		Hence, the isovolume leaf through $N_0$ coincides with the full connected component of $\Gr_S(M)$ containing $N_0$.
		For contractible curves, however, the isovolume leaves coincide with the isodrastic leaves according to Remark~\ref{R:null}.
	\end{example}

	\begin{example}[Disconnected curves in the 2-torus]\label{E:2d}
		We continue to consider the 2-torus $M=\mathbb R^2/\mathbb Z^2=S^1\times S^1$ with the standard volume form $\mu=dxdy$.
		Moreover, we assume $S=S^1\sqcup\cdots\sqcup S^1$ has $k$ connected components with $k\geq2$.
		According to \eqref{E:codim.Fcex} the isodrastic distribution has codimension $k$.
		It constrains the connected components of $N\in\Gr_S(M)$ independently (as long as they stay disjoint), cf.~Remark \ref{R:isodrast.cc}.
		As in the connected case discussed in Example~\ref{E:T2.conn}, some isovolume leaves are larger than the isodrastic leaves, 
		but the connected components of $N$ can not move independently while $N$ is constrained to an isovolume leaf.
		To see this consider $N_0=S^1\times\{y_1,\dotsc,y_k\}$ where $0\leq y_1<y_2<\cdots<y_k<1$.
		At $N_0$ the isovolume foliation has codimension $k-1$, see \eqref{E:codim.Fcmu}.
		The submanifold $S^1\times\{y_1+t,y_2+t,\dotsc,y_k+t\}$ lies in the isovolume leaf through $N_0$, for each $t$.
		Actually, Proposition~\ref{P:below} shows that a submanifold of the form $N=S^1\times\{y'_1,\dotsc,y'_k\}$ with $0\leq y'_1<y'_2<\cdots<y'_k<1$ 
		lies in the isovolume leaf through $N_0$ if and only if there exists $t$ such that $y'_i=y_i+t$ for all $i=1,\dotsc,k$.
	\end{example}
		
	Example~\ref{E:2d} can be extended to manifolds of the form $M=Q\times S^1$ endowed with a product volume form, where $Q$ is any orientable closed connected manifold. 
	The isovolume leaf through the non-connected submanifold $N_0=Q\times\{y_1,\dotsc,y_k\}$ behaves like the one described above.

\section{EPDiffvol dual pair for codimension one}\label{S:codim.one}

	We continue to consider the codimension one case, cf.~\eqref{E:codim1}.
	We present two dual pairs for the volume preserving diffeomorphism group, involving the isodrastic foliation and the isovolume foliation, respectively.

\subsection{Infinitesimal transitivity}

	Let $\mathcal E_\circ=T^*_{\reg,\circ}\Emb(S,M)$ denote the set of all pairs $(\varphi,\alpha)\in T^*_\reg\Emb(S,M)$ 
	for which $\varphi^*\alpha$ is a nowhere vanishing section of $|\Lambda|_S\otimes T^*S$.
	Clearly, $\mathcal E_\circ$ is an open subset of $\mathcal E$ that is invariant under $\Diff(M)$ and $\Diff(S)$.
	Note that $\mathcal E_\circ$ can only be nonempty if each connected component of $S$ has vanishing Euler characteristics.

	\begin{lemma}\label{L:zetaE0}
		Suppose $\Phi=(\varphi,\alpha)\in\mathcal E_\circ$ and $X\in\mathfrak X(M)$ such that $\varphi^*i_X\mu$ is exact.
		Then there exists $Y\in\mathfrak X_{c,\ex}(M)$ such that $\zeta^{\mathcal E_\circ}_X(\Phi)=\zeta^{\mathcal E_\circ}_Y(\Phi)$, 
		where $\zeta^{\mathcal E_\circ}$ denotes the infinitesimal $\Diff(M)$ action on $\mathcal E_\circ$.
	\end{lemma}

	This lemma follows immediately from the following

	\begin{lemma}\label{L:zetaT*M0}
		Let $N\subseteq M$ be a compact submanifold without boundary of codimension $\dim M-\dim N=1$, 
		suppose $\alpha\in\Gamma(T^*M|_N)$ projects to a nowhere vanishing section of $T^*N$, and consider $X\in\mathfrak X(M)$.
		\begin{enumerate}[(a)]
			\item	If $i_X\mu$ becomes exact when pulled back to $N$, 
				then there exists $Y\in\mathfrak X_{c,\ex}(M,\mu)$ with $\zeta^{T^*M}_X\circ\alpha=\zeta^{T^*M}_Y\circ\alpha$.
			\item	If the pull back of $i_X\mu$ to $N$ represents a cohomology class in the image of $H^*_c(M)\to H^*(N)$, 
				then there exists $Y\in\mathfrak X_c(M,\mu)$ with $\zeta^{T^*M}_X\circ\alpha=\zeta^{T^*M}_Y\circ\alpha$.
			\item	If the pull back of $i_X\mu$ to $N$ represents a cohomology class in the image of $H^*(M)\to H^*(N)$, 
				then there exists $Y\in\mathfrak X(M,\mu)$ with $\zeta^{T^*M}_X\circ\alpha=\zeta^{T^*M}_Y\circ\alpha$.
		\end{enumerate}
	\end{lemma}

	\begin{proof}
		As in the proof of Lemma~\ref{L:zetaT*M} we begin by constructing a vector field $Z\in\mathfrak X(M)$ such that
		\begin{equation}\label{E:tZ.var}
		        Z_x=0,\qquad\alpha(D_xZ)=0,\qquad\text{and}\qquad\tr(D_xZ)>0
		\end{equation}
		for all $x\in N$.
		To this end we fix a fiber-wise Euclidean inner product on $TM|_N$ and consider again $\tilde Z\in\Gamma(\eend(TM|_N))$ given by $\tilde Z=P_{\ker\alpha}P_{T^\perp N}$.
		Since the codimension is one, we now have
		\[
		        \tr(\tilde Z_x)=1-\frac{\alpha_x(\nu)^2}{|\alpha_x|^2}
		\]
		where $\nu$ is a unit normal to $T_xN$.
		In particular, $\tr(\tilde Z_x)>0$ in view of the assumption that $\alpha_x$ does not vanish on $T_xN$.
		As before, $\tilde Z$ gives rise to a fiber-wise linear vector field on the total space of the normal bundle $T^\perp N$ 
		which can be chopped off and glued into $M$ to obtain a vector field $Z$ on $M$ such that $Z_x=0$ and $D_xZ=\tilde Z_x$ for all $x\in N$, whence \eqref{E:tZ.var}.

		We choose a function $f$ on $M$ such that
		\[
			f|_N=\frac{(L_X\mu)|_N}{\tr(DZ|_N)\mu}.
		\]
		Proceeding as in the proof of Lemma~\ref{L:zetaT*M}, we see that the vector field $\tilde X:=X-fZ$ satisfies
		\begin{equation}\label{E:tX.var}
			\zeta^{T^*M}_{\tilde X}\circ\alpha=\zeta^{T^*M}_X\circ\alpha\qquad\text{and}\qquad(L_{\tilde X}\mu)|_N=0.
		\end{equation}

		Using the relative Poincar\'e lemma as in the proof of Lemma~\ref{L:zetaT*M}, we obtain a vector field $\tilde Y$ on an open neighborhood $U$ of $N$ such that
		\begin{equation}\label{E:tY.var}
		        (j^1\tilde Y)|_N=(j^1\tilde X)|_N\qquad\text{and}\qquad L_{\tilde Y}\mu=0.
		\end{equation}

		Note that the pull backs of $i_X\mu$ and $i_{\tilde Y}\mu$ to $N$ coincide.
		If they represent a class in the image of $H^*(M)\to H^*(N)$, then there exists a vector field $Y\in\mathfrak X(M,\mu)$ which coincides with $\tilde Y$ in a neighborhood of $N$.
		Combining this with \eqref{E:tX.var} and \eqref{E:tY.var}, we obtain $\zeta^{T^*M}_X\circ\alpha=\zeta^{T^*M}_Y\circ\alpha$.
		If the pull back of $i_X\mu$ to $N$ represents a class in the image of $H^*_c(M)\to H^*(N)$, then $Y$ may be chosen in $\mathfrak X_c(M,\mu)$.
		If the pull back of $i_X\mu$ to $N$ is exact, then $Y$ may be chosen in $\mathfrak X_{c,\ex}(M,\mu)$.
	\end{proof}

\subsection{EPDiffvol dual pair for isodrastic leaves}

	As in Section~\ref{SS:T*regL} we let $L$ denote the preimage under $\Emb(S,M)\to\Gr_S(M)$ of an isodrastic leaf in $\Gr_S(M)$, i.e., the preimage of a $\Diff_{c,\ex}(M,\mu)$ orbit in $\Gr_S(M)$.
	Recall that the connected components of $L$ are isodrastic leaves in $\Emb(S,M)$.
	Let 
	\begin{equation}\label{E:varL}
		\mathcal L:=q(\iota^{-1}(\mathcal E_\circ)){=T^*_{\reg,\circ}L\subseteq T^*_\reg L}
	\end{equation}
	denote the subset of all $(\varphi,[\alpha])\in T^*_\reg L$ such that $\alpha$ projects to a nowhere vanishing section of  $|\Lambda|_S\otimes T^*S$.
	Clearly, this is an open subset in $T^*_\reg L$ which is invariant under the action of $\Diff_{c,\ex}(M,\mu)$ and $\Diff(S)$.
	The restrictions to $\mathcal L$ of the canonical 1-form and the symplectic form on $T^*_\reg L$ will be denoted by $\theta^{\mathcal L}$ and $\omega^{\mathcal L}$, respectively.

	\begin{theorem}\label{T:dual.pair.codim.one}
		Let $\mu$ be a volume form on a smooth manifold $M$, and suppose $S$ is a closed manifold such that $\dim M-\dim S=1$.
		Moreover, let $L\subseteq\Emb(S,M)$ denote the preimage of an isodrastic leaf in $\Gr_S(M)$.
		Then the weak symplectic Fr\'echet manifold $\mathcal L{=T^*_{\reg,\circ}L}$ in \eqref{E:varL}, 
		together with the commuting actions of $\Diff_{c,ex}(M,\mu)$ and $\Diff(S)$, 
		and the equivariant moment maps associated with the canonical invariant 1-form $\theta^{\mathcal L}$, 
		\begin{equation}\label{E:dp.vol.codim.one}
			\mathfrak X_{c,\ex}(M,\mu)^*\xleftarrow{J^{\mathcal L}_L}\mathcal L\xrightarrow{J^{\mathcal L}_R}\Omega^1(S;|\Lambda|_S)\subseteq\mathfrak X(S)^*
		\end{equation}
		constitute a symplectic dual pair with mutually symplectic orthogonal orbits.
		More explicitly, the moment maps are
		\begin{equation}\label{E:moma.vol.codim.one}
			\langle J^{\mathcal L}_L(\varphi,[\alpha]),X\rangle=\int_S\alpha(X\circ\varphi)
			\qquad\text{and}\qquad
			J^{\mathcal L}_R(\varphi,[\alpha])=\varphi^*\alpha
		\end{equation}
		where $X\in\mathfrak X_{c,\ex}(M,\mu)$ and $(\varphi,[\alpha])\in\mathcal L$.
	\end{theorem}

	\begin{proof}
		Using \eqref{E:thetaL} we see that the restrictions of the moment maps in \eqref{E:dp.vol} to $\iota^{-1}(\mathcal E_\circ)$ 
		both factor through the projection $q\colon\iota^{-1}(\mathcal E_\circ)\to\mathcal L$ and give rise to the moment maps in \eqref{E:dp.vol.codim.one},
		see also \eqref{E:varL} and \eqref{E:diag}.
		The formulas for the moment maps provided in \eqref{E:moma.vol.codim.one} thus follow from the formulas in \eqref{E:moma.vol}.
		The fact that the expressions in \eqref{E:moma.vol.codim.one} do indeed only depend on the class represented by $\alpha$ can alternatively be seen using 
		Lemma~\ref{L:Fcex.int.Emb} which shows that $f\mu_\varphi$ vanishes on the tangent space $\mathcal D_{\varphi,\ex}$ of the isodrastic leaf through $\varphi$,
		and both $X\circ\varphi$ for $X\in\mathfrak X_{c,\ex}(M,\mu)$ and $T\varphi\circ Y$ for $Y\in\mathfrak X(S)$ belong to $\mathcal D_{\varphi,\ex}$, see also \eqref{explit}.

		Suppose $\xi\in T_{(\varphi,[\alpha])}\mathcal L$ is such that 
		\begin{equation}\label{E:p1}
			\omega^{\mathcal L}\bigl(\xi,\zeta^{\mathcal L}_Y(\varphi,[\alpha])\bigr)=0,\qquad\textrm{for all $Y\in\mathfrak X(S)$.}
		\end{equation}
		We have to show that $\xi=\zeta^{\mathcal L}_X(\varphi,[\alpha])$ for some vector field $X\in\mathfrak X_{c,\ex}(M,\mu)$.
		To this end, choose $\tilde\xi\in T_{(\varphi,\alpha)}(\mathcal E_\circ|_L)$ such that $Tq\cdot\tilde\xi=\xi$.	
		Note that $Tp\cdot\tilde\xi\in T_\varphi L$.
		By the $\Diff(S)$ equivariance of $q$ and \eqref{E:omegaL} we obtain from \eqref{E:p1}
		\[
			\omega^{\mathcal E_\circ}\bigl(\tilde\xi,\zeta^{\mathcal E_\circ}_Y(\varphi,\alpha)\bigr)=0,\qquad\textrm{for all $Y\in\mathfrak X(S)$.}
		\]
		As the $\Diff(S)$ and $\Diff_c(M)$ orbits on $\mathcal E_\circ$ are mutually symplectic orthogonal, see Theorem~\ref{epdiff} or \cite{GBV12},
		there exists $Z\in\mathfrak X_c(M)$ such that $\tilde\xi=\zeta_Z^{\mathcal E_\circ}(\varphi,\alpha)$.
		Since $p$ is $\Diff(M)$ equivariant this yields
		\[
			Tp\cdot\tilde\xi=\zeta_Z^{\Emb(S,M)}(\varphi)=Z\circ\varphi.
		\]
		As $Tp\cdot\tilde\xi$ is tangential to the isodrastic foliation, we conclude that $\varphi^*i_Z\mu$ is exact.
		Hence, by Lemma~\ref{L:zetaE0} there exists $X\in\mathfrak X_{c,\ex}(M,\mu)$ such that $\zeta^{\mathcal E_\circ}_X(\varphi,\alpha)=\zeta^{\mathcal E_\circ}_Z(\varphi,\alpha)$.
		We conclude $\tilde\xi=\zeta^{\mathcal E_\circ}_X(\varphi,\alpha)$ and then $\xi=Tq\cdot\tilde\xi=\zeta^{\mathcal L}_X(\varphi,[\alpha])$ since $q$ is $\Diff_{c,ex}(M,\mu)$ equivariant.
		Hence, $X$ has the desired property.
		
		Suppose conversely $\xi\in T_{(\varphi,[\alpha])}\mathcal L$ is such that 
		\begin{equation}\label{E:p2}
			\omega^{\mathcal L}\bigl(\xi,\zeta^{\mathcal L}_X(\varphi,[\alpha])\bigr)=0,\qquad\textrm{for all $X\in\mathfrak X_{c,\ex}(M,\mu)$.}
		\end{equation}
		We have to show that $\xi=\zeta^{\mathcal L}_Y(\varphi,[\alpha])$ for some vector field $Y\in\mathfrak X(S)$.
		To this end, choose $\tilde\xi\in T_{(\varphi,\alpha)}(\mathcal E_\circ|_L)$ such that $Tq\cdot\tilde\xi=\xi$.	
		By the $\Diff_{c,\ex}(M,\mu)$ equivariance of $q$ and \eqref{E:omegaL} we obtain from \eqref{E:p2}
		\[
			\omega^{\mathcal E_\circ}\bigl(\tilde\xi,\zeta^{\mathcal E_\circ}_X(\varphi,\alpha)\bigr)=0,\qquad\textrm{for all $X\in\mathfrak X_{c,\ex}(M,\mu)$.}
		\]
		Combining this with Lemma~\ref{L:zetaE0} we obtain
		\begin{equation}\label{E:p3}
			\omega^{\mathcal E_\circ}\bigl(\tilde\xi,\zeta^{\mathcal E_\circ}_X(\varphi,\alpha)\bigr)=0,\qquad
			\begin{array}{l}
				\textrm{for all $X\in\mathfrak X_c(M)$ such that}\\\textrm{$X\circ\varphi$ is tangential to $L$ at $\varphi$.}
			\end{array}
		\end{equation}
		Choose vector fields $X_1,\dotsc,X_r\in\mathfrak X_c(M)$ such that $X_1\circ\varphi,\dotsc,X_r\circ\varphi$ 
		induces a basis of the normal space $\bigl(T_\varphi\Emb(S,M)\bigr)|_L/T_\varphi L$.
		As $p$ is $\Diff(M)$ equivariant, $Tp\cdot\zeta^{\mathcal E_\circ}_{X_j}(\varphi,\alpha)=\zeta^{\Emb(S,M)}_{X_j}(\varphi)=X_j\circ\varphi$.
		In view of the nondegenerate pairing in \eqref{E:pairing} we conclude that there exists $\eta\in\ker(T_{(\varphi,\alpha)}q)$ such that
		\[
			\omega^{\mathcal E_\circ}\bigl(\eta,\zeta^{\mathcal E_\circ}_{X_j}(\varphi,\alpha)\bigr)
			=\omega^{\mathcal E_\circ}\bigl(\tilde\xi,\zeta^{\mathcal E_\circ}_{X_j}(\varphi,\alpha)\bigr),\qquad j=1,\dotsc,r.
		\]
		Hence, by replacing the lift $\tilde\xi$ with $\tilde\xi-\eta$, 
		we may assume $\omega^{\mathcal E_\circ}\bigl(\tilde\xi,\zeta^{\mathcal E_\circ}_{X_j}(\varphi,\alpha)\bigr)=0$ for $j=1,\dotsc,r$.
		Combining this with \eqref{E:p3} we obtain
		\[
			\omega^{\mathcal E_\circ}\bigl(\tilde\xi,\zeta^{\mathcal E_\circ}_X(\varphi,\alpha)\bigr)=0,\qquad\textrm{for all $X\in\mathfrak X_c(M)$.}
		\]
		Since the $\Diff(S)$ and $\Diff_c(M)$ orbits on $\mathcal E_\circ$ are mutually symplectic orthogonal, see Theorem~\ref{epdiff} or \cite{GBV12},
		there exists $Y\in\mathfrak X(S)$ such that $\tilde\xi=\zeta_Y^{\mathcal E_\circ}(\varphi,\alpha)$.
		As $q$ is $\Diff(S)$ equivariant this yields $\xi=Tq\cdot\tilde\xi=\zeta_Y^{\mathcal L}(\varphi,[\alpha])$.
		Hence, $Y$ has the desired property.
	\end{proof}

	Using Lemma~\ref{L:cove} it is easy to see that $\Ann_\reg(\mathcal D_\ex)$ is a vector subbundle of finite rank in $T^*_\reg\Gr_S(M)\subseteq\Gr^\aug_S(M)$.
	The quotient bundle,
	\begin{multline*}
		\Gr_S^{[\aug]}(M):=\Gr^\aug_S(M)/\Ann_\reg(\mathcal D_\ex)
		\\
		=\bigl\{(N,[\gamma]):N\in\Gr_S(M),\gamma\in\Gamma(|\Lambda|_N\otimes T^*M|_N)\bigr\},
	\end{multline*}
	consists of submanifolds $N$ augmented by equivalence classes of $\gamma$, taken modulo $\Ann_\reg(\mathcal D_{N,\ex})$.
	The space $\Gr^{[\aug]}_S(M)$ inherits the structure of a Fr\'echet manifold, and the map in \eqref{gtog} factors to a fiber bundle projection $\Gr^{[\aug]}_S(M)\to\Gr_S^\deco(M)$.
	The fiber over $(N,\rho_N)\in\Gr_S^\deco(M)$ is an affine space over the vector space $\Gamma(|\Lambda|_N\otimes\Ann(TN))/\Ann_\reg(\mathcal D_{N,\ex})$.
	Using Lemma~\ref{L:Fcex.int.Emb} we see that $\Ann_\reg(\mathcal D_\ex)\subseteq T^*_\reg\Emb(S,M)$ is a vector subbundle of finite rank too.
	We obtain a sequence of smooth locally trivial fiber bundles,
	\begin{equation}\label{E:augS}
		\frac{T^*_\reg\Emb(S,M)}{\Ann_\reg(\mathcal D_\ex)}\xrightarrow{\Diff(S)}\Gr_S^{[\aug]}(M)\xrightarrow{{\rm affine}}\Gr_S^\deco(M)\xrightarrow{\Omega^1(S;|\Lambda|_S)}\Gr_S(M),
	\end{equation}
	with structure group and typical fibers as indicated over the arrows.

	Restricting everything to an isodrastic leaf $I\subseteq\Gr_S(M)$ we obtain, cf.~\eqref{E:T*regL}, a sequence of smooth fiber bundles
	\begin{equation}\label{E:augSI}
		T^*_\reg L\xrightarrow{\Diff(S)}\Gr_{S,I}^{[\aug]}(M)\xrightarrow{{\rm affine}}\Gr_{S,I}^\deco(M)\xrightarrow{\Omega^1(S;|\Lambda|_S)}I,
	\end{equation}
	where $\Gr_{S,I}^{[\aug]}(M):=\{(N,[\gamma])\in\Gr^{[\aug]}_S(M):N\in I\}$ and $\Gr_{S,I}^\deco(M):=\{(N,\rho_N)\in\Gr_S^\deco(M):N\in I\}$.
	Clearly, 
	\[
		\mathcal M:=\Gr^{[\aug]}_{S,I,\circ}:=\{(N,[\gamma])\in\Gr^{[\aug]}_{S,I}(M):\text{$\iota_N^*\gamma$ nowhere zero}\bigr\}
	\]
	is an open subset in $\Gr_{S,I}^{[\aug]}(M)$.
	As $\mathcal L$ in \eqref{E:varL} is the preimage of $\mathcal M$ under the bundle projection $T^*_\reg L\to\Gr_{S,I}^{[\aug]}(M)$, 
	the latter restricts to a smooth principal $\Diff(S)$ bundle $\mathcal L\to\mathcal M$.
	The left moment map in \eqref{E:moma.vol.codim.one} descends to a $\Diff_{c,\ex}(M,\mu)$ equivariant injective map
	\begin{equation}\label{E:barJLL}
		\bar J^{\mathcal L}_L\colon\mathcal M\to\mathfrak X_{c,\ex}(M,\mu)^*,\qquad
		\langle\bar J^{\mathcal L}_L(N,[\gamma]),X\rangle=\int_N\gamma(X|_N)
	\end{equation}
	where $(N,[\gamma])\in\mathcal M$ and $X\in\mathfrak X_{c,\ex}(M,\mu)$.
	The injectivity follows from Lemma~\ref{L:Fcex.int}.

\subsection{EPDiffvol dual pair for isovolume leaves}

	Let $L\subseteq\Emb(S,M)$ denote the preimage of a $\Diff_c(M,\mu)$ orbit in $\Gr_S(M)$.
	The connected components of $L$ are isovolume leaves.
	We have commuting actions of $\Diff_c(M,\mu)$ and $\Diff(S)$ on $L$.
	We define the regular cotangent bundle $T^*_\reg L$ as in \eqref{E:T*regL}.
	More explicitly,
	\begin{equation}\label{E:T*regL.isovol}
		T_\reg^*L=\bigl\{(\varphi,[\alpha]):\varphi\in L,\alpha\in\Gamma(|\Lambda|_S\otimes\varphi^*T^*M)\bigr\}
	\end{equation}
	where $[\alpha_1]=[\alpha_2]$ if and only if $\alpha_2-\alpha_1=f\mu_\varphi$, cf.~\eqref{E:mu.iso}, with 
	\begin{multline*}
		f\in\Ann\bigl(\img\bigl(\varphi^*\colon H^{\dim S}_c(M;\mathbb R)\to H^{\dim S}(S;\mathbb R)\bigr)\bigr)
		\\
		\subseteq H^0(S;\mathfrak o_S)=\Gamma_\locconst(\mathfrak o_S),
	\end{multline*}
	where the annihilator is with respect to the canonical pairing 
	\[
		H^0(S;\mathfrak o_S)\times H^{\dim S}(S;\mathbb R)\to\mathbb R.
	\]
	The canonical one form $\theta^{T^*_\reg L}$ is invariant under the induced (commuting) actions of $\Diff_c(M,\mu)$ and $\Diff(S)$ on $T^*_\reg L$.
	As in the isodrastic case, $\omega^{T^*_\reg L}=d\theta^{T^*_\reg L}$ is a weakly nondegenerate symplectic form on $T^*_\reg L$.

	We let $T^*_{\reg,\circ}L$ denote the subset of $T^*_\reg L$ consisting of all pairs $(\varphi,[\alpha])$ 
	for which $\alpha$ projects to a nowhere vanishing section of $|\Lambda|_S\otimes T^*S$.
	Clearly, this is an open subset in $T^*_\reg L$ which is invariant under the action of $\Diff_c(M,\mu)$ and $\Diff(S)$.
	Proceeding as before and using Lemma~\ref{L:zetaT*M0}(b) we obtain the following variation of Theorem~\ref{T:dual.pair.codim.one}.

	\begin{theorem}\label{T:dual.pair.codim.one.var}
		Let $\mu$ be a volume form on a smooth manifold $M$, and suppose $S$ is a closed manifold such that $\dim M-\dim S=1$.
		Moreover, let $L\subseteq\Emb(S,M)$ denote the preimage of a $\Diff_c(M,\mu)$ orbit in $\Gr_S(M)$.
		Then the weak symplectic Fr\'echet manifold $\mathcal L=T^*_{\reg,\circ}L$, 
		together with the commuting actions of $\Diff_c(M,\mu)$ and $\Diff(S)$, 
		and the equivariant moment maps associated with the canonical invariant 1-form $\theta^{\mathcal L}$, 
		\begin{equation}\label{E:dp.isovol}
			\mathfrak X_c(M,\mu)^*\xleftarrow{J^{\mathcal L}_L}\mathcal L\xrightarrow{J^{\mathcal L}_R}\Omega^1(S;|\Lambda|_S)\subseteq\mathfrak X(S)^*,
		\end{equation}
		constitute a symplectic dual pair with mutually symplectic orthogonal orbits.
		More explicitly, the moment maps are
		\[
			\langle J^{\mathcal L}_L(\varphi,[\alpha]),X\rangle=\int_S\alpha(X\circ\varphi)
			\qquad\text{and}\qquad
			J^{\mathcal L}_R(\varphi,[\alpha])=\varphi^*\alpha
		\]
		where $X\in\mathfrak X_c(M,\mu)$ and $(\varphi,[\alpha])\in\mathcal L$.
	\end{theorem}
	
	\begin{example}[Plane curves]\label{E:pc}
		As in Example~\ref{E:R2.conn} we consider $M=\mathbb R^2$ endowed with $\mu=dxdy$.
		Recall that $\Diff_c(M)$ is connected \cite{S59} and, thus, $\Diff_c(M,\mu)$ is connected too.
		Hence, $\Diff_c(M,\mu)=\Diff_c(M,\mu)_0=\Diff_{c,\ex}(M,\mu)$ is the group of compactly supported Hamiltonian diffeomorphisms, with Lie algebra $C_c^\infty(\mathbb R^2)$. 
		Let  $S=S^1$. 
		Since the distributions $\mathcal D_\ex$ and $\mathcal D_\mu$ coincide, 
		the symplectic dual pairs in Theorems~\ref{T:dual.pair.codim.one} and \ref{T:dual.pair.codim.one.var} coincide in this case.
		Moreover, the preimage $L$ of an isodrastic leaf in the nonlinear Grassmannian is the manifold $\Emb^a(S^1,\mathbb R^2)$ of embeddings enclosing a fixed area $a>0$.
		Thus, we have a symplectic dual pair 
		\begin{equation}\label{E:mica}
			C_c^\infty(\mathbb R^2)^*\xleftarrow{J^{\mathcal L}_L}T^*_{\reg,\circ}\Emb^a(S^1,\mathbb R^2)
			\xrightarrow{J^{\mathcal L}_R}\Omega^1(S^1;|\Lambda|_{S^1})\subseteq\mathfrak X(S^1)^*
		\end{equation}
		for the natural actions of $\Ham_c(\mathbb R^2)$ and $\Diff(S^1)$ on $\mathcal L=T^*_{\reg,\circ}\Emb^a(S^1,\mathbb R^2)$.
		Here $L$ and $\mathcal L$ have two and four connected components, respectively.
	\end{example}

\section{Coadjoint orbits for codimension one}\label{S:orbits.codim.one}

	We perform symplectic reduction for the right leg of the dual pairs in Theorems~\ref{T:dual.pair.codim.one} and \ref{T:dual.pair.codim.one.var},
	to get coadjoint orbits of $\Diff_{c,\ex}(M,\mu)$ and $\Diff_c(M,\mu)$, respectively.
	Reduction at the level zero is not possible here, since the right moment maps in these dual pairs do not take the value zero.
	Therefore we will restrict, as in Section~\ref{SS:EPDiff.nonzero}, to 1-dimensional manifolds $S$ and nowhere vanishing $\rho$.
	We will see that this choice leads to a reduced symplectic space that can be endowed with a natural Fr\'echet manifold structure. 

	We are in the codimension one case, so $M$ is a surface and the volume form $\mu$ a symplectic form. 
	Hence, the group $\Diff_{c,\ex}(M,\mu)$ is the group of compactly supported Hamiltonian diffeomorphisms $\Ham_c(M)$ with Lie algebra $\ham_c(M)=C^\infty_c(M)$, 
	and the isodrastic foliation from Section~\ref{iso} is the original isodrastic foliation of Weinstein's \cite{W90}. 
 
\subsection{Augmented nonlinear Grassmannians over isodrasts of curves}\label{SS:61}

	Let $I$ be an isodrastic leaf, i.e., a $\Ham_c(M)$ orbit in $\Gr_{S}(M)$.
	Let $L$ denote the preimage of $I$ in $\Emb(S,M)$.
	We recall that the dual pair \eqref{E:dp.vol.codim.one} is defined on the open subset $\mathcal L=T^*_{\reg,\circ}L$ of the regular cotangent bundle $T^*_\reg L$ in \eqref{explit},
	characterized by $\varphi^*\alpha\in\Omega^1(S;|\Lambda|_S)$ being nowhere zero.
	Let $\rho\in\Omega^1(S;|\Lambda|_S)\subseteq\mathfrak X(S)^*$ be a nowhere zero 1-form density.
	The reduced symplectic manifold at $\rho$ along the right leg of the dual pair, $\mathcal M _\rho:=(J_R^\mathcal L)^{-1}(\mathcal O_\rho)/\Diff(S)$, can be identified with
	\begin{equation}\label{tw}
		\mathcal M _\rho=\bigl\{(N,[\gamma])\in\Gr_{S,I}^{[\aug]}(M):(N,\iota_N^*\gamma)\cong(S,\rho)\bigr\}.
	\end{equation}
	Clearly, $\mathcal M _\rho$ is the preimage of
	\[
		\Gr_{S,I,\rho}^\deco(M)
		:=\bigl\{(N,\rho_N)\in\Gr_{S,I}^\deco(M):(N,\rho_N)\cong(S,\rho)\bigr\}
	\]
	under the middle fiber bundle projection in \eqref{E:augSI}, see also \eqref{gero2}.
	Recall from Section~\ref{SS:EPDiff.nonzero} that $\mathcal O_\rho$ is a smooth submanifold of $\Omega^1(S;|\Lambda|_S)$.
	Restricting the identification in \eqref{gero} to the submanifold $I$ in $\Gr_S(M)$, we see that $\Gr_{S,I,\rho}^\deco(M)$ is a smooth submanifold in $\Gr_{S,I}^\deco(M)$.
	We obtain a sequence of smooth fiber bundles,
	\begin{equation}\label{E:2}
		(J^{\mathcal L}_R)^{-1}(\mathcal O_\rho)\xrightarrow{\Diff(S)}\mathcal M _\rho\xrightarrow{{\rm affine}}\Gr_{S,I,\rho}^\deco(M)\xrightarrow{\mathcal O_\rho}I
	\end{equation}
	with structure group and typical fibers as indicated.
	Each space in this sequence is a smooth submanifold in the corresponding term in \eqref{E:augSI}.

	The left moment map in \eqref{E:dp.vol.codim.one} restricts and descends to a $\Diff_{c,\ex}(M,\mu)$ equivariant injective map
	\begin{equation}\label{barj}
		\bar J^{\mathcal L}_L\colon\mathcal M _\rho\to\ham_c(M)^*=C^\infty_c(M)^*,\quad\langle\bar J^{\mathcal L}_L(N,[\gamma]),X\rangle=\int_N\gamma(X|_N),
	\end{equation}
	where $(N,[\gamma])\in\mathcal M_\rho$ and $X\in\ham_c(M)$.
	The injectivity follows from the injectivity of the map in \eqref{E:barJLL}. 
	Using Theorem~\ref{T:dual.pair.codim.one} and proceeding as in \cite[Proposition~2]{HV04} we see that the $\Ham_c(M)$ action on $\mathcal M _\rho$ is locally transitive.
	Moreover, the pullback of $\omega^\KKS$ coincides with the reduced symplectic form on $\mathcal M _\rho$.
	This shows

	\begin{corollary}\label{T:new}
		Given an isodrastic leaf $I\subseteq\Gr_S(M)$ and a nowhere zero $\rho\in\Omega^1(S;|\Lambda|_S)$, 
		then each connected component of the augmented nonlinear Grassmannian $\mathcal M _\rho$ in \eqref{tw},
		endowed with its reduced symplectic form and canonical $\Ham_c(M)$ action, 
		is symplectomorphic to a coadjoint orbit of the Hamiltonian group $\Ham_c(M)$
		via the equivariant moment map \eqref{barj}.
	\end{corollary}

	A similar statement remains true if the Hamiltonian group is replaced with the symplectic group $\Symp_c(M)=\Diff_c(M,\mu)$,
	and isodrastic leaves (orbits of $\Ham_c(M)$) by orbits of $\Symp_c(M)$.
	The reduced symplectic manifold at $\rho$ along the right leg of the dual pair in Theorem~\ref{T:dual.pair.codim.one.var}
	can be identified with the augmented nonlinear Grassmannian
	\begin{equation}\label{ni}
		\mathcal N_\rho:=\bigl\{(N,[\gamma]):N\in I,\gamma\in\Gamma(|\Lambda|_N\otimes T^*M|_N),(N,\iota_N^*\gamma)\cong(S,\rho)\bigr\},
	\end{equation}
	with $I$ denoting a $\Symp_c(M)$ orbit in $\Gr_S(M)$.
	This time the class of $\gamma$ is considered modulo $\Ann_\reg(\mathcal D_{N,\mu})$, i.e.,
	modulo the image under the isomorphism $\mu_N$ in \eqref{E:mu.Gr2} of the annihilator $\Ann(\iota_N^*H_c^1(M;\mathbb R))\subseteq H^0(N;\mathfrak o_N)$
	with respect to the canonical pairing $H^0(N;\mathfrak o_N)\times H^1(N;\mathbb R)\to\mathbb R$.
	This permits to show that $\Ann_\reg(\mathcal D_\mu)$ is a vector subbundle of finite rank in $T^*_\reg\Gr_S(M)$.
	One can then proceed exactly as in the isodrastic case to conclude that $\mathcal N _\rho$ is a Fr\'echet manifold in a natural way. 
	The left moment map in \eqref{E:dp.isovol} restricts and descends to a $\Symp_c(M)$ equivariant injective map
	\begin{equation}\label{barj3}
		\bar J^{\mathcal L}_L\colon\mathcal N _\rho\to\mathfrak X_c(M,\mu)^*,\qquad\langle\bar J^{\mathcal L}_L(N,[\gamma]),X\rangle=\int_N\gamma(X|_N),
	\end{equation}
	Using Theorem~\ref{T:dual.pair.codim.one.var} and proceeding as in \cite[Proposition~2]{HV04} we see that the $\Symp_c(M)$ action on $\mathcal N_\rho$ is locally transitive.
	Moreover, the pullback of $\omega^\KKS$ is the reduced symplectic form on $\mathcal N_\rho$.
	We thus obtain

	\begin{corollary}\label{T:new2}
		Given a $\Symp_c(M)$ orbit $I$ in $\Gr_S(M)$ and a nowhere zero $\rho\in\Omega^1(S;|\Lambda|_S)$, 
		then each connected component of $\mathcal N_\rho$ in \eqref{ni}, endowed with its reduced symplectic form and canonical $\Symp_c(M)_0$ action, 
		is symplectomorphic to a coadjoint orbit of the symplectomorphism group $\Symp_c(M)_0$ via the equivariant moment map in \eqref{barj3}.
		Moreover, this moment map identifies certain unions of connected components of $\mathcal N_\rho$ with coadjoint orbits of the full group, $\Symp_c(M)$.
	\end{corollary}

	\begin{example}[Connected curves in the plane]\label{vort}
		In the situation of Example~\ref{E:pc} the two corollaries above yield the same result.
		One obtains reduced symplectic manifolds $\mathcal M_\rho^a$ of augmented plane curves enclosing a fixed area $a>0$ 
		as coadjoint orbits of $\Ham_c(\mathbb R^2)=\Symp_c(\mathbb R^2)_0$. 
	\end{example}

	\begin{example}[Connected curves in the 2-torus]
		Applying Corollary~\ref{T:new2} to the isovolume leaf through the meridian $N_0$ in Example~\ref{E:T2.conn}, 
		we see that the corresponding connected components of $\mathcal G_\rho$ in \eqref{E:grh} are coadjoint orbits of $\Symp(T^2)_0$. 
	\end{example}

\subsection{Prequantization}

	In this section we present a sufficient condition for the prequantizability of the coadjoint orbits of augmented curves described in the Corollaries~\ref{T:new} and \ref{T:new2}. 
	Therefore, let $S=\bigsqcup_{i=1}^kS_i$ be a disjoint union of circles and $\rho$ a nowhere zero 1-form density on $S$.
	Let $\ell_\rho=(l_1,\dotsc,l_k)$ with $l_i$ the circle length measured with the restriction $\rho_i\in\Omega^1(S_i;|\Lambda|_{S_i})$.
		
	\begin{theorem}\label{P:preq}
		Suppose $l_i^2\in\mathbb N$ for $i=1,\dotsc,k$.
		Then the symplectic manifold $\mathcal M_\rho$ in \eqref{tw} is prequantizable, for every isodrastic leaf $I$ in $\Gr_S(M)$.
		Moreover, the symplectic manifold $\mathcal N_\rho$ in \eqref{ni} is prequantizable, for every $\Symp_c(M)$ orbit $I$ in $\Gr_S(M)$.
		In particular, the coadjoint $\Ham_c(M)$ orbits in Corollary~\ref{T:new} and the coadjoint $\Symp_c(M)$ orbits in Corollary~\ref{T:new2} are prequantizable, for such a $\rho$.
	\end{theorem}

	\begin{proof}
		Suppose $I$ is an isodrastic leaf in $\Gr_S(M)$ and let $L$ denote its preimage in $\Emb(S,M)$. 
		The principal $\Diff(S)$ bundle in \eqref{E:2} restricts to a principal $\Diff(S,\rho)$ bundle $p\colon\mathcal P_\rho\to\mathcal M_\rho$, where
		\[
			\mathcal P_\rho:=(J_R^{\mathcal L})^{-1}(\rho)=\{(\varphi,[\alpha])\in T^*_\reg L:\varphi^*\alpha=\rho\}\subseteq\mathcal L=T^*_{\reg,\circ}L.
		\]
		Moreover, the canonical 1-form $\theta^{\mathcal L}$ restricts to an invariant 1-form $\theta_\rho$ on $\mathcal P_\rho$ such that $d\theta_\rho=p^*\omega_\rho$,
		where $\omega_\rho$ denotes the reduced symplectic form on $\mathcal M_\rho$.
	
		Since $\Diff(S,\rho)_0\cong(S^1)^k$ is compact and abelian, the kernel of the exponential map $\exp\colon\mathfrak X(S,\rho)\to\Diff(S,\rho)$ is a lattice in $\mathfrak X(S,\rho)$.
		This lattice is generated by $Y_1,\dotsc,Y_k$ where $Y_i\in\mathfrak X(S_i,\rho_i)\subseteq\mathfrak X(S,\rho)$ denotes the vector field supported on $S_i$ 
		such that $L_{Y_i}\rho_i=0$ and $\int_{S_i}i_{Y_i}\rho=l_i^2$.
		Equivalently, this can be characterized as the positively oriented vector field of constant length $l_i$ on $S_i$, 
		with respect to the orientation and the Riemannian metric corresponding to $\rho_i$.
		Hence, for $(\varphi,[\alpha])\in\mathcal P_\rho$ we have
		\begin{equation}\label{repro}
			\theta_\rho\bigl(\zeta_{Y_i}(\varphi,[\alpha])\bigr)
			=\langle J_R^{\mathcal L}(\varphi,[\alpha]),Y_i\rangle
			=\langle\rho,Y_i\rangle
			=\int_Si_{Y_i}\rho
			=l_i^2,
		\end{equation}
		where $\zeta_Y$ denotes the infinitesimal principal action of $Y\in\mathfrak X(S,\rho)$ on $\mathcal P_\rho$. 
		
		The codimension one Lie subalgebra
		\[
			\mathfrak h:=\{Y\in\mathfrak X(S,\rho):\theta_\rho(\zeta_Y)=0\}
		\]
		is stable under the adjoint $\Diff(S,\rho)$ action, by invariance of $\theta_\rho$.
		Let $n\in\mathbb N$ be the greatest common divisor of $l_1^2,\dotsc,l_k^2$, which means that $l_i^2=nm_i$ for $m_1,\dotsc,m_k$ coprime integers.
		Together with \eqref{repro} this yields
		\[
			 \mathfrak h=\left\{\sum_{i=1}^kx_iY_i:x_i\in\mathbb R, m_1x_1+\cdots+m_kx_k=0\right\}.
		\]
		Hence, $\mathfrak h$ integrates to a closed connected Lie subgroup $H$ in $\Diff(S,\rho)_0$.
		Moreover, $H$ is normal in $\Diff(S,\rho)$ as $\mathfrak h$ is stable under this group.
		Using the description in \eqref{E:semidirect} we conclude that $\bar H:=H\rtimes\mathfrak S_k(\ell_\rho)$ is a normal subgroup in $\Diff(S,\rho)$ 
		which intersects each connected component, and the quotient group $\Diff(S,\rho)/\bar H$ is a circle.
		The kernel of $\exp\colon\mathfrak X(S,\rho)/\mathfrak h\to\Diff(S,\rho)/\bar H$ is a rank one lattice spanned by 
		\[
			\bar Y:=\tfrac1{m_1}Y_1+\mathfrak h=\cdots=\tfrac1{m_k}Y_k+\mathfrak h.
		\] 
		Indeed, this kernel consists of all elements of the form $\sum_{i=1}^kn_iY_i+\mathfrak h$ with $n_i\in\mathbb Z$. 
		These can be rewritten as $\bigl(\frac{1}{m_1}\sum_{i=1}^kn_im_i\bigr)Y_1+\mathfrak h$.
		Since $m_1,\dotsc,m_k$ are coprime integers, there exist $n_1,\dotsc,n_k\in\mathbb Z$ such that $\sum_{i=1}^k m_in_i=1$.
		Hence, the kernel coincides with $\mathbb Z\bar Y$.
		
		With \eqref{repro} we get $\theta_\rho(\zeta_{\bar Y})=n$.
		Put $\tilde Y:=\frac1n\bar Y$ and let $\tilde H$ denote the subgroup of $\Diff(S,\rho)$ generated by $\bar H$ and $\exp(\tilde Y)$.
		This is a normal subgroup, the quotient group $\Diff(S,\rho)/\tilde H$ is a circle, 
		the kernel of the exponential map $\exp\colon\mathfrak X(S,\rho)/\mathfrak h\to\Diff(S,\rho)/\tilde H$ is spanned by $\tilde Y$, and $\theta_\rho(\zeta_{\tilde Y})=1$. 
		We conclude that the quotient manifold $\tilde{\mathcal P}_\rho:=\mathcal P_\rho/\tilde H$ is a principal circle bundle over $\mathcal M_\rho$ 
		with structure group $\Diff(S,\rho)/\tilde H\cong S^1$.
		Moreover, the canonical 1-form $\theta_\rho$ descends to a principal connection 1-form on $\tilde{\mathcal P}_\rho$ with curvature $\omega_\rho$.
		
		An analogous proof yields the prequantization of $\mathcal N_\rho$, for every $\Symp_c(M)$ orbit $I$ in $\Gr_S(M)$.
		The prequantizability of the coadjoint orbits now follows from Corollaries~\ref{T:new} and \ref{T:new2}.
	\end{proof}

\subsection{Singular vortex configurations}\label{SS:63}

	The vorticity of an ideal fluid, whether a regular or a singular vorticity, is confined to a coadjoint orbit of the group of volume preserving diffeomorphisms.
	In the plane, the group of compactly supported volume preserving diffeomorphisms coincides with $\Ham_c(\mathbb R^2)$, and its Lie algebra is $C^\infty_c(\mathbb R^2)$
	(by identifying the Hamiltonian vector field $X_h$ with its compactly supported Hamiltonian function).
	Examples of singular vortex configurations in the plane are the point vortex $x\in\mathbb R^2$ with pairing $\langle x,X_h\rangle=h(x)$ 
	and the (weighted) vortex loop $(C,|\nu|)$ with pairing $\langle(C,|\nu|),X_h\rangle=\int_Ch|\nu|$, where $|\nu|$ is a density on the closed curve $C$,
	which can be viewed as a loop of point vortices.

	A vortex dipole in the plane is a vector $u_x\in T_x\mathbb R^2$ with pairing $\langle u_x,X_h\rangle=d_xh(u_x)$.
	Its loop version would be a triple $(C,|\nu|,u)$ consisting of an embedded closed curve $C$ endowed with a density $|\nu|$, 
	and a vector field $u$ in the plane along $C$, with pairing $\langle(C,|\nu|,u),X_h\rangle=\int_Cdh(u)|\nu|$.
	To obtain an identification of triples with elements in the dual Lie algebra, one has to use of a tensor product between $|\nu|$ and $u$, 
	as well as a factorization of $u$ by all vector fields on $C$ whose contraction with $\nu$ is constant.
	A way to solve this is to integrate with the density defined by $u$, namely $|\nu|=|\iota_C^*i_u\mu|$:
	\begin{equation}\label{lopdip}
		\langle(C,[u]),X_h\rangle=\int_Cdh(u)|\iota_C^*i_u\mu|,
	\end{equation}
	where the class $[u]$ is considered modulo $\mathbb R t_C$ with $t_C$ denoting the vector field on $C$ characterized by $\mu(u,t_C)=1$.
	For this to make sense we restrict to $u$ that are nowhere tangent to $C$.

	Next we show that the coadjoint orbits of loops of vortex dipoles coincide with those described in Section~\ref{SS:61}.
	We consider $S=S^1$ and $M=\mathbb R^2$ endowed with canonical area form $\mu$.
	The isodrastic leaves in $\Gr_{S^1}(\mathbb R^2)$ are the nonlinear Grassmannians $\Gr_{S^1}^a(\mathbb R^2)$ 
	consisting of all closed curves that enclose a fixed area $a>0$ in the plane (see Example~\ref{E:pc}). 
	The nonlinear Grassmannian $\mathcal M_\rho$ of augmented closed curves in \eqref{tw}, obtained by symplectic reduction at a nowhere zero 1-form density 
	$\rho\in\Omega^1(S^1;|\Lambda|_{S^1})$ of length $\ell$, becomes (see Example~\ref{vort})
	\begin{equation}\label{marh}
		\mathcal M^a_\ell=\left\{(C,[\gamma]):
		\begin{array}{c}
			C\in\Gr_{S^1}^a(\mathbb R^2),\ \gamma\in\Gamma(|\Lambda|_C\otimes T^*\mathbb R^2|_C)\\
			\iota_C^*\gamma\text{ nowhere zero 1-form density of length }\ell
		\end{array}\right\}.
	\end{equation}
	Here $[\gamma]$ considered modulo the annihilator $\Ann_\reg(T_C\Gr_{S^1}^a(\mathbb R^2))=\mathbb R\varepsilon_C$, where
	\begin{equation}\label{tace}
		\varepsilon_C=|\nu_C|\otimes i_{t_C}\mu\in \Gamma(|\Lambda|_C\otimes\Ann(TC))
	\end{equation}
	for a volume form $\nu_C$ and its dual vector field $t_C$, i.e., $\nu_C(t_C)=1$.

	To extract $u\in\Gamma(T\mathbb R^2|_C)$ from $\gamma$, we use the 1-form density $\iota_C^*\gamma$ on $C$. 
	It is of type $\rho$, hence it encodes a metric of length $\ell$ on $C$ together with an orientation. 
	Hence, denoting by $\nu_C$ the volume form induced by the metric and compatible with the orientation, we have $\iota_C^*\gamma=|\nu_C|\otimes\nu_C$.
	There exists a unique vector field $u$ along $C$ such that
	\[
		\gamma=|\nu_C|\otimes i_u\mu\in\Gamma(|\Lambda|_C\otimes T^*\mathbb R^2|_C).
	\]
	Moreover, $\iota_C^*(i_u\mu)=\nu_C$ and $u$ is nowhere tangent to $C$.
	Thus, we get an identification of $\gamma$ with $u$,
	as well as an identification of $[\gamma]=\gamma+\mathbb R\varepsilon_C$ with $[u]=u+\mathbb Rt_C$.

	Since there is no point on the curve where the vector field $u\in\Gamma(T\mathbb R^2|_C)$ is tangent to the curve,
	either $u$ points inward at every point of $C$, or it points outward.
	This corresponds to the two connected components of $\mathcal M^a_\ell$, namely the two coadjoint orbits of $\Ham_c(\mathbb R^2)$.
	Also the orientations on $C$ induced by $\nu_C$ are opposite in these two cases.

	\begin{corollary}
		The symplectic reduction for the right action of $\Diff(S^1)$ on the cotangent bundle $T^*_{\reg,\circ}\Emb^a(S^1,\mathbb R^2)$
		at a nowhere zero 1-form density of length $\ell$ yields the nonlinear Grassmannian of augmented curves 
		\[
			\mathcal M^a_\ell=\left\{(C,[u]):
			\begin{array}{c}
				C\in\Gr_{S^1}^a(\mathbb R^2),\ u\in\Gamma(T\mathbb R^2|_C)\\
				\nu_C=\iota_C^*i_u\mu\text{ volume form of length }\ell
			\end{array}\right\},
		\]
		where the class $[u]$ is considered modulo $\mathbb Rt_C$.
		Endowed with its reduced symplectic form, each of the two connected components of $\mathcal M^a_\ell$ is symplectomorphic to a coadjoint orbit of $\Ham_c(\mathbb R^2)$ via
		\[
			J\colon\mathcal M^a_\ell\to C^\infty_c(M)^*,
			\quad\langle J(C,[u]),X_h\rangle=\int_Cdh(u)|\iota_C^*i_u\mu|.
		\]
		If, moreover, $\ell^2\in\mathbb N$, then both coadjoint orbits are prequantizable.
	\end{corollary}

	The first part of the corollary is a special case of Corollary~\ref{T:new}, while the second part follows directly from Theorem~\ref{P:preq}.

\section*{Acknowledgments}

	The first author thanks the West University of Ti\-mi\-\c soa\-ra for the warm hospitality.
	He gratefully acknowledges the support of the Austrian Science Fund (FWF), Grant DOI: 10.55776/P31663.
	The second author thanks the University of Vienna for the warm hospitality.
	She gratefully acknowledges the support of the grant of the Romanian Ministry of Education and Research, 
	CNCS-UEFISCDI, project number PN-III-P4-ID-PCE-2020-2888, within PNCDI III.


\begin{thebibliography}{XX}

\bibitem{BFP}
	J. Balog, L. Feh\'er and L. Palla,
	\textit{Coadjoint orbits of the Virasoro algebra and the global Liouville equation.} 
	Internat. J. Mod. Phys. A \textbf{13}(1998), 315--362.

\bibitem{cv}
	I. Ciuclea and C. Vizman, 
	\textit{Pointed vortex loops in ideal 2D fluids.}
	J. Phys. A \textbf{56}(2023), Paper No. 245201, 15 pp.

\bibitem{GBV12}
	F. Gay-Balmaz and C. Vizman, 
	\textit{Dual pairs in fluid dynamics.} 
	Ann. Global Anal. Geom. \textbf{41}(2012), 1--24.

\bibitem{GBV15}
	F. Gay-Balmaz and C. Vizman,
	\textit{A dual pair for free boundary fluids.}
	Int. J. Geom. Methods Mod. Phys. \textbf{12}(2015), Paper No. 1550068, 18 pp.

\bibitem{GBV19}
	F. Gay-Balmaz and C. Vizman,
	\textit{Isotropic submanifolds and coadjoint orbits of the Hamiltonian group.} 
	J. Symp. Geom. \textbf{17}(2019), 663--702.

\bibitem{GBV23}	
	F. Gay-Balmaz and C. Vizman,
	\textit{Coadjoint orbits of vortex sheets in ideal fluids.} 
	J. Geom. Phys. \textbf{197}(2024), Paper No. 105096, 13 pp.

\bibitem{goldin1}
	G. A. Goldin, R. Menikoff and D. H. Sharp,
	\textit{Diffeomorphism groups and quantized vortex filaments.}
	Phys. Rev. Lett. \textbf{58}(1987), 2162--2164.

\bibitem{goldin2}
	G. A. Goldin, R. Menikoff and D. H. Sharp,
	\textit{Quantum vortex configurations in three dimensions.}
	Phys. Rev. Lett. \textbf{67}(1991), 3499--3502.

\bibitem{HV04}
	S. Haller and C. Vizman,
	\textit{Non-linear Grassmannians as coadjoint orbits.}
	Math. Ann. \textbf{329}(2004), 771--785.

\bibitem{HV20}
	S. Haller and C. Vizman,
	\textit{Nonlinear flag manifolds as coadjoint orbits.}
	Ann. Global Anal. Geom. \textbf{58}(2020), 385--413.

\bibitem{HV22}
	S. Haller and C. Vizman,
	\textit{A dual pair for the contact group.} 
	Math. Z. \textbf{301}(2022), 2937--2973.

\bibitem{HV23}
	S. Haller and C. Vizman,
	\textit{Weighted nonlinear flag manifolds as coadjoint orbits.}
	Canad. J. Math., 2023, to appear in print.
	DOI \url{https://doi.org/10.4153/S0008414X23000585}

\bibitem{H76}
	M. W. Hirsch,
	\textit{Differential topology.}
	Grad. Texts in Math. \textbf{33}.
	Springer-Verlag, New York-Heidelberg, 1976.

\bibitem{HM05}
	D. D. Holm and J. E. Marsden,
	\textit{Moment maps and measure-valued solutions (peakons, filaments, and sheets) for the EPDiff equation.}
	In \textit{The breadth of symplectic and Poisson geometry}, volume \textbf{232} of Progr. Math., pages 203--235. 
	Birkh\"auser Boston, Boston, MA, 2005.

\bibitem{I96}
	R. S. Ismagilov, 
	\textit{Representations of Infinite-Dimensional Groups.}
	Translated from the Russian Manuscript by D. Deart., Translations of Mathematical Monographs, vol. \textbf{152}. 
	American Mathematical Society, Providence (1996).

\bibitem{izosimov}	
	A. Izosimov, B. Khesin and M. Mousavi,
	\textit{Coadjoint orbits of symplectic diffeomorphisms of surfaces and ideal hydrodynamics.}
	Annales de l'Institut Fourier \textbf{66}(2016), 2385--2433.

\bibitem{khesin}
	B. Khesin,
	\textit{Symplectic structures and dynamics on vortex membranes.}
	Moscow Math. J. \textbf{12}(2013), 413--434.

\bibitem{KM97}
	A. Kriegl and P. W. Michor.
	\textit{The convenient setting of global analysis.}
	Mathematical Surveys and Monographs \textbf{53}.
	American Mathematical Society, Providence, RI, 1997.

\bibitem{L09}
	B. Lee, 
	\textit{Geometric structures on spaces of weighted submanifolds.} 
	SIGMA \textbf{5}(2009), Article No. 099, 49pp.

\bibitem{LM87}
	P. Libermann and C.-M. Marle,
	\textit{Symplectic geometry and analytical mechanics.}
	D. Reidel Publishing Company, 1987.

\bibitem{MMOPR07}
	J. Marsden, G. Misio\l ek, J.-P. Ortega, M. Perlmutter, and T. Ratiu,
	\textit{Hamiltonian reduction by stages.}
	Lecture Notes in Math. \textbf{1913}.
	Springer, Berlin, 2007.
	
\bibitem{MW83}
	J. E. Marsden and A. Weinstein, 
	\textit{Coadjoint orbits, vortices, and Clebsch variables for incompressible fluids.}
	Phys. D \textbf{7}(1983), 305--323.
	
\bibitem{M65}
	J. Moser, 
	\textit{On the volume elements on a manifold.}
	Trans. Amer. Math. Soc. \textbf{120}(1965), 286--294.
	
\bibitem{PS}
	A. Pressley and G. Segal,
	\textit{Loop groups.} 
	Oxford Mathematical Monographs, Oxford University Press, 1986.

\bibitem{S59}
	S. Smale, 
	\textit{Diffeomorphisms of the 2-sphere.}
	Proc. Amer. Math. Soc. \textbf{10}(1959), 621--626.

\bibitem{W83}
	A. Weinstein,
	\textit{The local structure of Poisson manifolds.}
	J. Differential Geom. \textbf{18}(1983), 523--557.

\bibitem{W90}
	A. Weinstein,
	\textit{Connections of Berry and Hannay type for moving Lagrangian submanifolds.} 
	Adv. Math. \textbf{82}(1990), 133--159.

\end{thebibliography}
\end{document}